\documentclass[11pt,a4paper]{amsart}

\usepackage{a4wide}

\usepackage[english]{babel}

\usepackage[pdftex]{graphicx}

\usepackage[all]{xy}
\usepackage{amsmath,amsthm,amssymb,enumerate}
\usepackage{latexsym}
\usepackage{amscd}

\usepackage[colorlinks=true]{hyperref}

\newtheorem{theorem}{Theorem}[section]
\newtheorem{thm}[theorem]{Theorem}
\newtheorem{defn}[theorem]{Definition}
\newtheorem{prop}[theorem]{Proposition}

\newtheorem{lemma}[theorem]{Lemma}
\newtheorem{lem}[theorem]{Lemma}

\newtheorem{remark}[theorem]{Remark}
\newtheorem{rem}[theorem]{Remark}
\newtheorem{example}[theorem]{Example}

\newcommand{\cat}{\mathcal}

\newcommand{\R}{\mathbb R}
\newcommand{\Q}{\mathbb Q}
\newcommand{\Z}{\mathbb Z}
\newcommand{\N}{\mathbb N}
\newcommand{\C}{\mathbb C}
\newcommand{\T}{\mathbb T}
\newcommand{\CP}{\mathbb P}

\newcommand{\ft}{{\mathfrak{t}}}
\newcommand{\fk}{{\mathfrak{k}}}
\newcommand{\fg}{{\mathfrak{g}}}

\newcommand{\Xx}{{\cat X}}
\newcommand{\Ll}{{\cat L}}
\newcommand{\Mm}{{\cat M}}
\newcommand{\Nn}{{\cat N}}

\newcommand{\Pp}{{\cat P}}
\newcommand{\Dd}{{\cat D}}

\newcommand{\om}{{\omega}}
\newcommand{\vr}{{\varphi}}
\newcommand{\eps}{{\varepsilon}}

\newcommand{\p}{{\partial}}
\newcommand{\cp}{{{\C\CP}\,\!}}
\newcommand{\oz}{{\bar{z}}}

\newcommand{\tgamma}{{\tilde{\gamma}}}
\newcommand{\teta}{{\tilde{\eta}}}
\newcommand{\tR}{{\tilde{R}}}

\DeclareMathOperator{\sign}{sign}
\DeclareMathOperator{\rank}{rank}

\begin{document}
\title{Contact homology of good toric contact manifolds}

\author[M.~Abreu]{Miguel Abreu}
\address{Centro de An\'{a}lise Matem\'{a}tica, Geometria e Sistemas
Din\^{a}micos, Departamento de Matem\'atica, Instituto Superior T\'ecnico, Av.
Rovisco Pais, 1049-001 Lisboa, Portugal}
\email{mabreu@math.ist.utl.pt}
 
\author[L.~Macarini]{Leonardo Macarini}
\address{Universidade Federal do Rio de Janeiro, Instituto de Matem\'atica,
Cidade Universit\'aria, CEP 21941-909 - Rio de Janeiro - Brazil}
\email{leonardo@impa.br}

\date{April 12, 2011}

\begin{abstract}
In this paper we show that any good toric contact manifold has well defined
cylindrical contact homology and describe how it can be combinatorially computed
from the associated moment cone. As an application we compute the cylindrical
contact homology of a particularly nice family of examples that appear in the
work of Gauntlett-Martelli-Sparks-Waldram on Sasaki-Einstein metrics. We show
in particular that these give rise to a new infinite family of non-equivalent
contact structures on $S^2 \times S^{3}$ in the unique homotopy class of 
almost contact structures with vanishing first Chern class.
\end{abstract}



\thanks{Partially supported by Funda\c c\~ao para a 
Ci\^encia e a Tecnologia (FCT/Portugal), Funda\c c\~ao 
Coordena\c c\~ao de Aperfei\c coamento de Pessoal de N\'{\i}vel Superior 
(CAPES/Brazil) and Conselho Nacional de Desenvolvimento Cient\'{\i}fico 
e Tecnol\'ogico (CNPq/Brazil).}

\maketitle

\section{Introduction}
\label{s:intro}

Contact homology is a powerful invariant of contact structures, introduced by
Eliashberg, Givental and Hofer~\cite{EGH} in the bigger framework of Symplectic
Field Theory. Its simplest version is called cylindrical contact homology and can
be briefly described in the following way. Let $(N,\xi)$ be a closed (i.e. compact
without boundary) 
co-oriented contact manifold, $\alpha\in\Omega^1(N)$ a contact form 
($\xi = \ker \alpha$) and $R_\alpha\in\Xx(N)$ the corresponding Reeb vector field 
($\iota (R_\alpha) d\alpha \equiv 0$ and $\alpha (R_\alpha) \equiv 1$).
Assume that $\alpha$ is nondegenerate, with the meaning that all contractible closed orbits
of $R_\alpha$ are nondegenerate. Consider the graded $\Q$-vector space 
$C_\ast (N,\alpha)$ freely generated by the contractible closed orbits of $R_\alpha$, 
where the grading is determined by an appropriate dimensional shift of the 
Conley-Zehnder index (when the first Chern class of the contact structure is zero this
grading is integral, but otherwise it is jus a finite cyclic grading). 
One then uses suitable pseudo-holomorphic curves in the symplectization of $(N, \alpha)$ 
to define a linear map 
$\partial:C_\ast (N,\alpha) \to C_{\ast-1} (N,\alpha)$. Under suitable assumptions, one 
can prove that $\partial^2 = 0$ and the homology of $(C_\ast (N,\alpha), \partial)$
is independent of the choice of contact form $\alpha$. This is the cylindrical contact
homology $HC_\ast(N,\xi; \Q)$, a graded $\Q$-vector space invariant of the contact
manifold $(N, \xi)$. (Note: for transversality reasons, the identity $\partial^2 = 0$ and
the invariance property of contact homology are conditional on the completion of foundational 
work by Hofer, Wysocki and Zehnder.)

The simplest example of a cylindrical contact homology computation, already 
described in~\cite{EGH}, is the case of the standard contact sphere 
$(S^{2n+1}, \xi_{\rm st})$, where 
\[
S^{2n+1} \cong \{ z\in\C^{n+1} : \sum_{j=1}^{n+1} |z_j|^2 = 1 \} \subset \C^{n+1} 
\]
and
\[ 
\xi_{\rm st} := T S^{2n+1} \cap \,i\, T S^{2n+1} = \ 
\text{hyperplane field of complex tangencies.}
\]
This contact structure admits the natural contact form
\[
\alpha_{\rm st} := \frac{i}{2} \sum_{j=1}^{n+1} 
(z_j d \oz_j - \oz_j d z_j)|_{S^{2n+1}}
\]
with completely periodic Reeb flow given by
\[
(R_{\rm st})_s (z_1, \ldots, z_{n+1}) \mapsto (e^{is} z_1, \ldots, e^{is} z_{n+1})\,,\ 
s\in\R\,.
\]
In the presence of a degenerate contact form, such as $\alpha_{\rm st}$, there
are two approaches to compute cylindrical contact homology:
\begin{itemize}
\item[(i)] the Morse-Bott approach of Bourgeois~\cite{Bo}, that computes it 
directly from the spaces of contractible periodic orbits of the degenerate 
Reeb flow;
\item[(ii)] the nondegenerate approach above, that computes it from the
countably many contractible periodic orbits of the Reeb flow associated to 
a nondegenerate perturbation of the original degenerate contact form.
\end{itemize}
Since (ii) is the approach we will use in this paper, let us give 
a description of how it can work in this $(S^{2n+1}, \xi_{\rm st})$ example.
One can obtain a suitable perturbation of the contact form $\alpha_{\rm st}$
by perturbing the embedding of $S^{2n+1} \hookrightarrow \C^{n+1}$ via
\[
S^{2n+1} \cong S^{2n+1}_a := 
\{ z\in\C^{n+1} : \sum_{j=1}^{n+1} a_j |z_j|^2 = 1 \} \subset \C^{n+1}\,,\ 
\text{with $a_j \in\R^+$ for all $j=1, \ldots, n+1$,}
\]
and noting that
\[ 
\xi_{\rm st} \cong \xi_a :=  T S^{2n+1}_a \cap \,i\, T S^{2n+1}_a \,.
\]
The perturbed contact form $\alpha_a$ is then given by
\[
\alpha_a := \frac{i}{2} \sum_{j=1}^{n+1} (z_j d \oz_j - \oz_j d z_j)|_{S^{2n+1}_a}
\]
and the corresponding Reeb flow can be written as
\[
(R_a)_s (z_1, \ldots, z_{n+1}) \mapsto (e^{ia_1 s} z_1, \ldots, e^{i a_{n+1} s} 
z_{n+1})\,,\ s\in\R\,.
\]
If the $a_j$'s are $\Q$-independent, the contact form $\alpha_a$ is nondegenerate
and the Reeb flow has exactly $n+1$ simple closed orbits 
$\gamma_1, \ldots, \gamma_{n+1}$, where each $\gamma_\ell$ corresponds to the
orbit of the Reeb flow through the point $p_\ell \in S^{2n+1}_a$ with coordinates
\[
z_j = 
\begin{cases}
1/\sqrt{a_\ell} &\text{if $j=\ell$;} \\
0        &\text{if $j\ne\ell$.}
\end{cases} 
\]
As it turns out, in this example any closed Reeb orbit $\gamma_\ell^N$ has even
contact homology degree, which implies that the boundary operator is zero and
\[
HC_\ast (S^{2n+1}, \xi_{\rm st}; \Q) = C_\ast (\alpha_a)\,.
\]
After some simple Conley-Zehnder index computations one concludes that
\[
HC_\ast (S^{2n+1}, \xi_{\rm st}; \Q) \cong
\begin{cases}
\Q & \text{if $\ast \geq 2n$ and even;} \\
0 & \text{otherwise.}
\end{cases} 
\]

The standard contact sphere $(S^{2n+1}, \xi_{\rm st})$ is the most basic example 
of a good toric contact manifold and, as we show in this paper, this contact 
homology calculation has a toric description that generalizes to any good toric
contact manifold.

Closed toric contact manifolds are the odd dimensional analogues of
closed toric symplectic manifolds. They can be defined as contact
manifolds of dimension $2n+1$ equipped with an effective Hamiltonian action
of a torus of dimension $n+1$ and have been classified by
Banyaga-Molino~\cite{BnM1, BnM2, Bn}, Boyer-Galicki~\cite{BG} and 
Lerman~\cite{Le}.

Good toric contact manifolds of dimension three are $(S^3, \xi_{\rm st})$
and its finite quotients. Good toric contact manifolds of dimension greater
than three are closed toric contact manifolds whose torus action is not free. 
These form the most important class of closed toric contact manifolds and 
can be classified by the associated moment cones, in the same way that 
Delzant's theorem classifies closed toric symplectic manifolds by the 
associated moment polytopes.

In this paper we show that any good toric contact manifold has well defined
cylindrical contact homology and describe how it can be combinatorially 
computed from the associated moment cone.

To be more precise, consider the following definition.
\begin{defn}
A nondegenerate contact form is called \emph{nice} if its Reeb flow
has no closed contractible orbit of contact homology degree equal to $-1$, 
$0$ or $1$. A nondegenerate contact form is called \emph{even} if all closed 
contractible orbits of its Reeb flow have even contact homology degree.
\end{defn}
The following proposition is a direct generalization to even contact forms
of a well-known result for nice contact forms.
\begin{prop} \label{prop:even-1}
Let $(N,\xi)$ be a contact manifold with an even or nice nondegenerate 
contact form $\alpha$. Then the boundary operator
$\partial:C_\ast (N,\alpha) \to C_{\ast-1} (N,\alpha)$ satisfies $\partial^2 = 0$ 
and the homology of $(C_\ast (N,\alpha), \partial)$ is independent of the choice 
of even or nice nondegenerate contact form $\alpha$. Hence, the cylindrical 
contact homology $HC_\ast(N,\xi; \Q)$ is a well defined invariant of the contact
manifold $(N,\xi)$.
\end{prop}

Our first main result is the following theorem.

\begin{thm} \label{thm:main1}
Any good toric contact manifold admits even nondegenerate toric contact forms.
The corresponding cylindrical contact homology, isomorphic to the chain complex
associated to any such contact form, is a well-defined invariant that
can be combinatorially computed from the associated good moment cone.
\end{thm}

By applying this theorem to a particularly nice family of examples, originally
considered by the mathematical physicists J.~Gauntlett, D.~Martelli, J.~Sparks 
and D.~Waldram, in the context of their work on Sasaki-Einstein 
metrics~\cite{GW1} (see also~\cite{MSY} and \cite{Ab}), we obtain the 
second main result of this paper.

\begin{thm} \label{thm:main2}
There are infinitely many non-equivalent contact structures $\xi_k$ on 
$S^2 \times S^3$, $k\in\N_0$, in the unique homotopy class determined by the 
vanishing of the first Chern class. These contact structures are toric and can 
be distinguished by the degree zero cylindrical contact homology. More precisely:
\[
\rank HC_\ast (S^2  \times S^3, \xi_k ; \Q) =
\begin{cases}
k & \text{if $\ast = 0$;} \\
2k+1  & \text{if $\ast = 2$;} \\
2k+2 & \text{if $\ast > 2$ and even;} \\
0 & \text{otherwise.}
\end{cases} 
\]
\end{thm}

\begin{rem}
All even nondegenerate toric contact forms that we consider for this family of
examples have exactly four simple closed Reeb orbits. As we will see in
Section~\ref{s:examples}, by a suitable choice of such contact forms it is
possible to concentrate all relevant contact homology information in the
multiples of a single simple closed Reeb orbit: the one with minimal
action or, equivalently, minimal period.
\end{rem}

\begin{rem}
Otto van Koert constructs in~\cite{vK2} an infinite family of non-equivalent 
contact structures on $S^2 \times S^3$ with vanishing first Chern class. Since 
his contact structures have vanishing degree zero contact homology, they are 
necessarily different from the ones given by the $k>0$ cases of 
Theorem~\ref{thm:main2}. We will see that $(S^2  \times S^3, \xi_0)$ is 
contactomorphic to the unit cosphere bundle of $S^3$.

A recent preprint by J.~Pati~\cite{Pa} discusses a generalization of the Morse-Bott
approach of Bourgeois to compute the contact homology of $S^1$-bundles over
certain symplectic orbifolds and applies it to toric contact manifolds. His
explicit examples do not overlap with the ones in this paper.

Note also the recent preprint by M.J.D.~Hamilton~\cite{Ha} discussing inequivalent
contact structures on simply-connected $5$-manifolds which arise as $S^1$-bundles
over simply-connected $4$-manifolds. His contact structures have non-zero first
Chern class.
\end{rem}

The paper is organized as follows. Section~\ref{s:tscones} contains the necessary 
introduction to toric contact manifolds, their classification and main properties. 
The Conley-Zehnder index is described in Section~\ref{index}, where we also give a 
proof of its invariance property under symplectic reduction by a circle action
(Lemma~\ref{lem:main}). This result, which plays an important role in the paper 
and could also be of independent interest, is known to experts but we could not 
find a reference. Section~\ref{s:cch} gives a more detailed description of cylindrical contact 
homology and contains a proof of Proposition~\ref{prop:even-1}
(restated there as Proposition~\ref{prop:even-2}). Section~\ref{s:proof1} contains 
the proof of Theorem~\ref{thm:main1}, while the examples and cylindrical contact 
homology computations relevant for Theorem~\ref{thm:main2} are the subject of 
Section~\ref{s:examples}.

\subsection*{Notation} In this paper, unless explicitly stated otherwise,
closed Reeb orbit means contractible closed Reeb orbit.

\subsection*{Acknowledgements} We thank Urs Frauenfelder, Gustavo Granja, 
Viktor Ginzburg, Umberto Hryniewicz, Otto van Koert and David Martinez for several useful 
discussions regarding this paper. We also thank an anonymous referee for corrections and useful
suggestions. 

We thank IMPA and IST for the warm hospitality during the preparation of this work.

These results were first presented by the first author at the Workshop on Conservative
Dynamics and Symplectic Geometry, IMPA, Rio de Janeiro, Brazil, August 3--7, 2009. He
thanks the organizers for the opportunity to participate in such a wonderful event.

\section{Toric contact manifolds} 
\label{s:tscones}

In this section we introduce toric contact manifolds via toric 
symplectic cones and describe their classification and explicit 
construction via the associated moment cones. We will also describe 
the fundamental group and the first Chern class of a toric symplectic 
cone, as well as the space of toric contact forms and Reeb vector 
fields that are relevant in this context.
For further details see Lerman's papers~\cite{Le,Le1,Le2}.

\subsection{Symplectic cones}

\begin{defn} \label{def:scone}
A \emph{symplectic cone} is a triple $(W, \omega, X)$, where 
$(W,	\omega)$ is a connected symplectic manifold, i.e. 
$\omega\in\Omega^2(W)$ is a closed and nondegenerate $2$-form, 
and $X\in\Xx (W)$ is a vector field generating a proper 
$\R$-action $\rho_t:W\to W$, $t\in\R$, such that
$\rho_t^\ast (\omega) = e^{2t} \omega$. Note that the 
\emph{Liouville vector field} $X$ satisfies
$\Ll_X \omega = 2\omega$, or equivalently
\[
\omega = \frac{1}{2} d (\iota(X) \omega)\,.
\]
A \emph{closed} symplectic cone is a symplectic cone
$(W, \omega, X)$ for which the quotient $W/\R$ is closed.
\end{defn}

\begin{defn} \label{def:ccontact}
A \emph{co-orientable contact manifold} is a pair $(N, \xi)$, where
$N$ is a connected odd dimensional manifold and $\xi \subset TN$ 
is an hyperplane distribution globally defined by $\xi = \ker \alpha$
for some $\alpha\in\Omega^1(N)$ such that $d\alpha|_\xi$ is
non-degenerate. Such a $1$-form $\alpha$ is called a \emph{contact form} 
for $\xi$ and the non-degeneracy condition is equivalent to $\xi$ being
\emph{maximally non-integrable}, i.e. its integral submanifolds have at most
half of its dimension.

A \emph{co-oriented} contact manifold is a triple $(N, \xi, [\alpha])$,
where $(N,\xi)$ is a co-orientable contact manifold and $[\alpha]$ is the
conformal class of some contact form $\alpha$, i.e.
\[
[\alpha] = \left\{e^h \alpha \ |\ h\in C^\infty(N) \right\}\,.
\]
\end{defn}

Given a co-oriented contact manifold $(N, \xi, [\alpha])$, with contact
form $\alpha$, let
\[
W:=N\times\R\,,\ \omega := d(e^t \alpha) \quad\text{and}\quad
X:= 2\frac{\partial}{\partial t}\,,
\]
where $t$ is the $\R$-coordinate. Then $(W,\omega,X)$ is a 
symplectic cone, usually called the \emph{symplectization} of
$(N, \xi, [\alpha])$.

Conversely, given a symplectic cone $(W, \omega, X)$ let 
\[
N := W/\R\,,\ \xi := \pi_\ast (\ker (\iota(X) \omega))
\quad\text{and}\quad 
\alpha := s^\ast (\iota(X) \omega)\,,
\]
where $\pi:W\to N$ is the natural principal $\R$-bundle quotient 
projection and $s:N\to W$ is any global section (note that such global 
sections always exist, since any principal $\R$-bundle is trivial). 
Then $(N,\xi, [\alpha])$ is a  co-oriented contact manifold whose symplectization 
is the symplectic cone $(W, \omega, X)$.

In fact, we have that
\begin{center}
co-oriented contact manifolds $\overset{1:1}{\longleftrightarrow}$ 
symplectic cones
\end{center}
(see Chapter 2 of~\cite{Le1} for details).
Under this bijection, closed contact manifolds correspond to closed symplectic cones 
and toric contact manifolds correspond to toric symplectic cones (see below). Moreover,
the following are equivalent:
\begin{itemize}
\item[(i)] choice of a contact form for $(N, \xi, [\alpha])$;
\item[(ii)] choice of a global section of $\pi:W\to N$;
\item[(iii)] choice of an $\R$-equivariant splitting $W\cong N\times\R$.
\end{itemize}

The choice of a contact form $\alpha$ for a contact manifold $(N,\xi)$ gives rise
to the \emph{Reeb vector field} $R_\alpha \in\Xx(N)$, uniquely defined by
\[
\iota(R_\alpha) d\alpha \equiv 0
\quad\text{and}\quad
\alpha (R_\alpha) \equiv 1 \,,
\]
and corresponding \emph{Reeb flow} $(R_\alpha)_s : N \to N$ satisfying
\[
(R_\alpha)^\ast_s (\alpha) = \alpha\,,\ \forall\,s\in\R\,.
\]
The obvious horizontal lift of $R_\alpha$ to the symplectic cone 
$(W=N\times\R, \omega = d(e^t \alpha),X = 2\frac{\partial}{\partial t})$
will also be denoted by $R_\alpha$. It satisfies
\[
[R_\alpha, X] = 0 \quad\text{and}\quad \iota(R_\alpha)\omega = - d (e^t)\,.
\]
In other words, the lift of the Reeb flow is $X$-preserving and Hamiltonian.

\begin{remark}
On a symplectic cone $(M,\omega,X)$, any $X$-preserving symplectic action of
a Lie group $G$ is Hamiltonian. In fact, the map $\mu : M \to \fg^\ast$ defined by
\[
\langle \mu , Y \rangle = \omega (X, Y_M)\,,\ \forall\, Y\in\fg\,,
\]
where $Y_M$ is the vector field on $M$ induced by $Y$ via the $G$-action,
is a moment map~\cite{Le1}.
\end{remark}

\begin{remark} \label{rmk:chern-class}
Any co-oriented contact manifold $(N, \xi, [\alpha])$ has well defined Chern
classes 
\[
c_k (\xi) \in H^{2k} (N;\Z)\,,\ k=1, \ldots,n\,,
\] 
given by the Chern classes of the conformal symplectic vector bundle
\[
(\xi, [d\alpha|_\xi]) \longrightarrow N\,.
\]
Under the canonical isomorphism $\pi^\ast : H^\ast (N;\Z) \to H^\ast (W, \Z)$,
induced by the natural principal $\R$-bundle projection $\pi:W\to N$, these 
Chern classes coincide with the Chern classes of the tangent bundle of the
symplectization $(W, \omega, X)$. In fact
\[
(TW, \omega) \cong \varepsilon^2 \oplus \pi^\ast (\xi, [d\alpha|_\xi])\,,
\]
where $\varepsilon^2$ is a trivial rank-$2$ symplectic vector bundle. The
choice of a contact form $\alpha$ gives rise to an explicit isomorphism
\[
\varepsilon^2 \cong \text{span} \{X, R_\alpha\} \quad\text{and}\quad
\pi^\ast (\xi) \cong (\text{span} \{X, R_\alpha\})^\omega\,.
\]
\end{remark}

\begin{example} \label{ex:R}
The most basic example of a symplectic cone is $\R^{2(n+1)}\setminus\{0\}$ with 
linear coordinates 
\[
(u_1, \ldots, u_{n+1}, v_1, \ldots, v_{n+1})\,,
\]
symplectic form  
\[
\om_{\rm st} = du \wedge dv := \sum_{j=1}^{n+1} du_j \wedge dv_j
\]
and Liouville vector field
\[
X_{\rm st} = u\frac{\p}{\p u} +  v\frac{\p}{\p v}
:= \sum_{j=1}^{n+1} \left( u_j\frac{\p}{\p u_j} +  v_j\frac{\p}{\p v_j}\right)\,.
\]
The associated co-oriented contact manifold is isomorphic to $(S^{2n+1}, \xi_{\rm st})$,
where $S^{2n+1} \subset \C^{n+1}$ is the unit sphere and $\xi_{\rm st}$ is the hyperplane 
distribution of complex tangencies, i.e.
\[
\xi_{\rm st} = T S^{2n+1} \cap i \, T S^{2n+1}\,.
\]
The restriction of $\alpha_{\rm st} : = \iota (X_{\rm st}) \om_{\rm st}$ to $S^{2n+1}$ is 
a contact form  for $\xi_{\rm st}$. Its Reeb flow $(R_{\rm st})_s$ is the restriction to 
$S^{2n+1}$ of the diagonal flow on $\C^{n+1}$ given by
\[
(R_{\rm st})_s \cdot (z_1,\ldots, z_{n+1}) = 
(e^{i s}z_1, \ldots, e^{i s}z_{n+1})\,,
\]
where
\[
z_j = u_j + i v_j\,,\ j=1,\ldots, n+1\,,
\]
give the usual identification $\R^{2(n+1)}\cong\C^{n+1}$.
\end{example}

\begin{example} \label{ex:boothby-wang}
Let $(M,\omega)$ be a symplectic manifold such that the cohomology class 
\[
\frac{1}{2\pi}[\omega] \in H^2(M,\R) \ 
\text{is integral, i.e. in the image of the natural map $H^2(M,\Z) \to H^2(M, \R)$.}
\] 
Suppose that $H^2(M,\Z)$ has no torsion, so that the above natural map is injective and
we can consider $H^2(M,\Z) \subset H^2 (M,\R)$. Denote by $\pi:N\to M$ the principal 
circle bundle with first Chern class 
\[
c_1 (N) = \frac{1}{2\pi}[\omega]\,.
\] 
A theorem of Boothby and Wang~\cite{BW} asserts that there is a connection $1$-form 
$\alpha$ on $N$ with $d\alpha = \pi^\ast\omega$ and, consequently, 
$\alpha$ is a contact form. We will call $(N, \xi:=\ker(\alpha))$ the \emph{Boothby-Wang} 
manifold of $(M, \omega)$. The associated symplectic cone is the total space of the 
corresponding line bundle $L\to M$ with the zero section deleted.
The Reeb vector field $R_\alpha$ generates the natural
$S^1$-action of $N$, associated to its circle bundle structure.

When $M=\cp^n$, with its standard Fubini-Study symplectic form, we recover 
Example~\ref{ex:R}, i.e. $(N,\xi) \cong (S^{2n+1}, \xi_{\rm st})$ and
$\pi:S^{2n+1}\to \cp^n$ is the Hopf map.
\end{example}

\subsection{Toric symplectic cones}

\begin{defn} \label{def:tscone}
A \emph{toric symplectic cone} is a symplectic cone $(W,\omega,X)$ of 
dimension $2(n+1)$ equipped with an effective $X$-preserving symplectic
$\T^{n+1}$-action, with moment map $\mu:W\to \ft^\ast \cong \R^{n+1}$ 
such that $\mu(\rho_t(w)) = e^{2t} \rho_t (w)\,,\ \forall\, w\in W,\, 
t\in\R$. Its \emph{moment cone} is defined to be the set
\[
C := \mu(W) \cup \{0\} \subset \R^{n+1}\,.
\]
\end{defn}

\begin{example} \label{ex:toric-R}
Consider the usual identification $\R^{2(n+1)}\cong\C^{n+1}$ given by
\[
z_j = u_j + i v_j\,,\ j=1,\ldots, n+1\,,
\]
and the standard $\T^{n+1}$-action defined by
\[
(y_1, \ldots, y_{n+1}) \cdot (z_1,\ldots, z_{n+1}) = 
(e^{i y_1}z_1, \ldots, e^{i y_{n+1}}z_{n+1})\,.
\]
The symplectic cone $(\R^{2(n+1)}\setminus\{0\}, \om_{\rm st}, X_{\rm st})$
of Example~\ref{ex:R} equipped with this $\T^{n+1}$-action is a toric symplectic 
cone. The moment map $\mu_{\rm st}:\R^{2(n+1)}\setminus\{0\} \to \R^{n+1}$ is 
given by
\[
\mu_{\rm st} (u_1, \ldots, u_{n+1}, v_1, \ldots, v_{n+1}) = 
\frac{1}{2} (u_1^2 + v_1^2, \ldots, u_{n+1}^2 + v_{n+1}^2)\,.
\]
and the moment cone is 
$
C = (\R_0^+)^{n+1} \subset \R^{n+1}\,.
$
\end{example}

In~\cite{Le} Lerman completed the classification of closed toric 
symplectic cones, initiated by Banyaga and Molino~\cite{BnM1,BnM2,Bn} 
and continued by Boyer and Galicki~\cite{BG}. The ones that are relevant 
for toric K\"ahler-Sasaki geometry are characterized by having good 
moment cones.

\begin{defn}[Lerman] \label{def:gcone}
A cone $C\subset\R^{n+1}$ is \emph{good} if there exists a minimal set 
of primitive vectors $\nu_1, \ldots, \nu_d \in \Z^{n+1}$, with 
$d\geq n+1$, such that
\begin{itemize}
\item[(i)] $C = \bigcap_{j=1}^d \{x\in\R^{n+1}\,:\ 
\ell_j (x) := \langle x, \nu_j \rangle \geq 0\}$.
\item[(ii)] any codimension-$k$ face of $C$, $1\leq k\leq n$, 
is the intersection of exactly $k$ facets whose set of normals can be 
completed to an integral base of $\Z^{n+1}$.
\end{itemize}
\end{defn}

\begin{theorem}[Banyaga-Molino, Boyer-Galicki, Lerman] \label{thm:gcone}
For each good cone $C\subset\R^{n+1}$ there exists a unique closed toric 
symplectic cone $(W_C, \om_C, X_C, \mu_C)$ with moment cone $C$.
\end{theorem}
\begin{defn} \label{def:good}
The closed toric symplectic cones (resp. closed toric contact manifolds)
characterized by Theorem~\ref{thm:gcone} will be called \emph{good} toric 
symplectic cones 
(resp. \emph{good} toric contact manifolds).
\end{defn}
\begin{remark} \label{rmk:good}
According to Lerman's classification (see Theorem 2.18 in~\cite{Le}),
the list of closed toric contact manifolds that are \emph{not good} is
the following:
\begin{itemize}
\item[(i)] certain overtwisted contact structures on $3$-dimensional
lens spaces (including $S^1\times S^2$);
\item[(ii)] the tight contact structures $\xi_n$, $n\geq 1$, on 
$\T^3 = S^1 \times \T^2$, defined as
\[
\xi_n = \ker (\cos(n\theta) dy_1 + \sin(n\theta) dy_2)\,,\ 
(\theta,y_1,y_2)\in S^1 \times \T^2
\]
(Giroux~\cite{Gi} and Kanda~\cite{Ka} proved independently that these
are all inequivalent);
\item[(iii)] a unique toric contact structure on each principal
$\T^{n+1}$-bundle over the sphere $S^n$, with $n\geq 2$.
\end{itemize}
Item (iii) classifies all closed toric contact manifolds of dimension
$2n+1$, $n\geq 2$, and free $\T^{n+1}$-action (\cite{Le}). Hence, a closed toric 
contact manifold of dimension greater than three
is good if and only if the corresponding torus action is not free.
\end{remark}

\begin{example} \label{ex:stdcone}

Let $P\subset\R^n$ be an \emph{integral Delzant polytope}, 
i.e. a Delzant polytope with integral vertices or, equivalently,
the moment polytope of a closed toric symplectic manifold 
$(M_P, \om_P, \mu_P)$ such that $\frac{1}{2\pi}[\omega]\in H^2(M_P, \Z)$.
Then, its \emph{standard cone}
\[
C:= \left\{z(x,1)\in\R^{n}\times\R\,:\ x\in P\,,\ z\geq 0\right\}
\subset\R^{n+1}
\]
is a good cone. Moreover
\begin{itemize}
\item[(i)] the toric symplectic manifold $(M_P, \om_P, \mu_P)$ is the 
$S^1\cong \{{\bf 1}\}\times S^1 \subset \T^{n+1}$ \emph{symplectic reduction} 
of the toric symplectic cone $(W_{C}, \omega_{C}, X_{C}, \mu_{C})$ 
(at level one).
\item[(ii)] $(N_C:=\mu_C^{-1}(\R^{n}\times\{1\}), 
\alpha_C := (\iota(X_C) \omega_C)|_{N_C})$ is the \emph{Boothby-Wang} 
manifold of $(M_P, \om_P)$. The restricted $\T^{n+1}$-action makes it a 
\emph{toric contact manifold}.
\item[(iii)] $(W_C, \omega_C, X_C)$ is the \emph{symplectization} of 
$(N_C, \alpha_C)$.
\end{itemize}
See Lemma 3.7 in~\cite{Le3} for a proof of these facts.

If $P\subset\R^n$ is the standard simplex, i.e. $M_P = \cp^n$, 
then its standard cone $C\subset\R^{n+1}$ is the moment cone of
$(W_C = \C^{n+1}\setminus\{0\}, \om_{\rm st}, X_{\rm st})$
equipped with the $\T^{n+1}$-action given by
\begin{align}
& (y_1,\ldots,y_n,y_{n+1})\cdot (z_1, \ldots, z_n, z_{n+1}) \notag \\
=\ & (e^{i(y_1+y_{n+1})}z_1, \ldots, e^{i(y_n+y_{n+1})}z_n,
e^{iy_{n+1}}z_{n+1})\,.\notag
\end{align}
The moment map $\mu_C : \C^{n+1}\setminus\{0\} \to \R^{n+1}$ is
given by
\[
\mu_{C} (z) = 
\frac{1}{2} (|z_1|^2, \ldots, |z_n|^2, |z_1|^2+\cdots+|z_n|^2
+ |z_{n+1}|^2)
\]
and 
\[
N_C := \mu_C^{-1}(\R^{n}\times\{1\}) = \left\{ z\in\C^{n+1}\,:\ 
\|z\|^2 = 2 \right\} \cong S^{2n+1}\,.
\]
\end{example}

\begin{remark} \label{rem:orbBW}
Up to a possible twist of the action by an automorphism of the torus
$\T^{n+1}$, any good toric symplectic cone can be obtained via an 
orbifold version of the Boothby-Wang construction of 
Example~\ref{ex:boothby-wang}, where the base is a toric symplectic
orbifold. 
\end{remark}

\subsection{Explicit Models}
\label{ss:models}

Like for closed toric symplectic manifolds, the existence part of  
Theorem~\ref{thm:gcone} follows from an explicit symplectic reduction 
construction, starting from a standard 
$(\R^{2d}\setminus\{0\}, \om_{\rm st}, X_{\rm st})$
(cf. Example~\ref{ex:toric-R}). Since it will be
needed later, we will briefly describe it here. Complete details can be
found, for example, in~\cite{Le} (proof of Lemma 6.3).

Let $C\subset(\R^{n+1})^\ast$ be a good cone defined by
\[
C = \bigcap_{j=1}^d \{x\in(\R^{n+1})^\ast\,:\ 
\ell_j (x) := \langle x, \nu_j \rangle \geq 0\}\,
\]
where $d\geq n+1$ is the number of facets and each $\nu_j$ is a primitive 
element of the lattice $\Z^{n+1} \subset \R^{n+1}$ (the inward-pointing 
normal to the $j$-th facet of $C$).

Let $(e_1, \ldots, e_d)$ denote the standard basis of $\R^d$, and define
a linear map $\beta : \R^d \to \R^{n+1}$ by 
\begin{equation} \label{def:beta}
\beta(e_j) = \nu_j\,,\ j=1,\ldots,d\,. 
\end{equation}
The conditions of Definition~\ref{def:gcone} imply that
$\beta$ is surjective. Denoting by $\fk$ its kernel, we have short
exact sequences
\[
0 \to \fk \stackrel{\iota}{\to} \R^d \stackrel{\beta}{\to}
\R^{n+1} \to 0
\ \ \ \mbox{and its dual}\ \ \ 
0 \to (\R^{n+1})^\ast \stackrel{\beta^\ast}{\to} (\R^d)^\ast 
\stackrel{\iota^\ast}{\to}\fk^\ast \to 0\ .
\]
Let $K$ denote the kernel of the map from $\T^d = \R^d/2\pi\Z^d$ to
$\T^{n+1} = \R^{n+1}/2\pi\Z^{n+1}$ induced by $\beta$. More precisely,
\begin{equation} \label{eq:K}
K = \left\{ [y]\in \T^d\,:\ \sum_{j=1}^{d} y_j \nu_j
\in 2\pi\Z^n\right\}\,.
\end{equation}
It is a compact abelian subgroup of $\T^d$ with Lie algebra 
$\fk = \ker (\beta)$. Note that $K$ need not be connected (this will be
relevant in the proof of Proposition~\ref{prop:pi1}).

Consider $\R^{2d}$ with its standard symplectic form
\[
\om_{\rm st} = du\wedge dv = \sum_{j=1}^d du_j\wedge dv_j
\]
and identify $\R^{2d}$ with $\C^d$ via $z_j = u_j + i v_j\,,\ 
j=1,\ldots,d$. The standard action of $\T^d$ on $\R^{2d}\cong
\C^d$ is given by
\[
y \cdot z = \left( e^{i y_1} z_1, \ldots, e^{i y_d} z_d\right)
\]
and has a moment map given by
\[
\phi_{\T^d} (z_1,\ldots,z_d) = \sum_{j=1}^d \frac{|z_j|^2}{2}\, e_j^\ast 
\in (\R^d)^\ast\,.
\]
Since $K$ is a subgroup of $\T^d$, $K$ acts on $\C^d$ with moment map
\begin{equation}\label{def:phiK}
\phi_K = \iota^\ast \circ \phi_{\T^d} =
\sum_{j=1}^d \frac{|z_j|^2}{2} \iota^\ast(e_j^\ast)\in \fk^\ast\ .
\end{equation}

The toric symplectic cone $(W_C,\om_C, X_C)$ associated to the good
cone $C$ is the symplectic reduction of 
$(\R^{2d}\setminus\{0\}, \om_{\rm st}, X_{\rm st})$ with respect to the
$K$-action, i.e.
\[
W_C = Z / K\ \ \mbox{where}\ \ Z=\phi_K^{-1}(0) \setminus\{0\}
\equiv\ \mbox{zero level set of moment map in $\R^{2d}\setminus\{0\}$,}
\]
the symplectic form $\om_C$ comes from $\om_{\rm st}$ via symplectic 
reduction, while the $\R$-action of the Liouville vector field $X_C$ and 
the action of $\T^{n+1} \cong \T^d/K$ are induced by the actions of 
$X_{\rm st}$ and $\T^d$ on $Z$.

\subsection{Fundamental group and first Chern class}

Lerman showed in~\cite{Le2} how to compute the fundamental group of
a good toric symplectic cone, which is canonically isomorphic to the
fundamental group of the associated good toric contact manifold.

\begin{prop} \label{prop:pi1} (\cite{Le2})
Let $W_C$ be the good toric symplectic cone determined by a good
cone $C\subset\R^{n+1}$. Let $\Nn := \Nn\{\nu_1, \ldots, \nu_d\}$
denote the sublattice of $\Z^{n+1}$ generated by the primitive integral 
normal vectors to the facets of $C$. The fundamental group of $W_C$ 
is the finite abelian group
\[
\Z^{n+1}/\Nn\,.
\]
\end{prop}

\begin{proof}(Outline)
\begin{itemize}
\item[(i)] We know that
\[
W_C = Z / K\,,
\]
where $K\subset\T^d$ acts on $\C^d$ with moment map
$\phi_K : \C^d \to \fk^\ast$ defined by~(\ref{def:phiK})
and $Z = \phi_K^{-1} (0)\setminus\{0\}$.
\item[(ii)] The set $Z$ has the homotopy type of
\[
\C^d \setminus (V_1 \cup \cdots \cup V_r)\,,
\]
where each $V_j \subset \C^d$ is a linear subspace of complex
codimension at least $2$. In particular,
\[
\pi_0 (Z) = \pi_1 (Z) = \pi_2 (Z) = 1\,.
\]
\item[(iii)] $K$ acts freely on $Z$ and the long
exact sequence of homotopy groups for the fibration
\[
K \to Z \to W_C
\]
implies that
\[
\pi_1 (W_C) = \pi_0 (K)\,.
\]
\item[(iv)] The fact that $K = \ker \beta$, with $\beta:\T^d \to \T^{n+1}$
defined by~(\ref{def:beta}), implies that
\[
\pi_0 (K) = \Z^{n+1} / \Nn \,.
\]
\end{itemize}
\end{proof}

Recall from Remark~\ref{rmk:chern-class} that the Chern classes of the 
tangent bundle of a symplectic cone can be canonically identified with
the Chern classes of the associated co-oriented contact manifold.
The following proposition gives a combinatorial characterization of
the vanishing of the first Chern class of good toric symplectic cones.

\begin{prop} \label{prop:c_1}
Let $(W_C, \omega_C,X_C)$ be the good toric symplectic cone determined
by the good cone $C\in\R^{n+1}$ via the explicit symplectic reduction
construction of the previous subsection. Let $K\in\T^d$ be defined
by~(\ref{eq:K}) and denote by $\chi_1, \ldots, \chi_d$ the characters
that determine its natural representation on $\C^d$. Then
\[
c_1 (TW_C) = 0 \ \Leftrightarrow\ \chi_1 + \ldots + \chi_d = 0\,.
\]
\end{prop}

\begin{proof}
It follows from the symplectic reduction description of $W_C$ as
$Z/K$, where $K\subset\T^d$ acts freely on $Z\subset\C^d$, that the
quotient map $Z\to W_C$ is a principal $K$-bundle and we have
the following classifying diagram:
\[ 
\xymatrix{  
Z  \ar[r] \ar[d]  & EK  \ar[d] \\ 
W_C \ar[r]^{f}  & BK  
} 
\]
Consider the induced map between the homotopy long exact sequences of these
two principal fibrations with the same fiber $K$. Note that $EK$ is contractible.
As we pointed out in the proof of the previous proposition, 
Lerman showed in~\cite{Le2} that $\pi_0 (Z) = \pi_1 (Z) = \pi_2 (Z) = 1$.
This implies that
\[
f_\ast : \pi_i (W_C) \longrightarrow \pi_i (BK) \ 
\text{is an isomorphism for $i=0, 1, 2$.}
\]
Since $\pi_3 (BK) \cong \pi_2 (K)  = 1$, we also know that
\[
f_\ast : \pi_3 (W_C) \longrightarrow \pi_3 (BK) \ 
\text{is surjective.}
\]
This means that the map $f: W_C \to BK$ is $3$-connected and so induces an isomorphism
in homology, and also in cohomology, in degree $\leq 2$. In particular,
\[
f^\ast : H^2 (BK;\Z) \longrightarrow H^2 (W_C; \Z) \ 
\text{is an isomorphism.}
\]

The natural representation of $K\subset \T^d$ on $\C^d$ and this principal $K$-bundle
$Z \to W_C$ give rise to a vector bundle $Z \times_K \C^d \to W_C$ with the following
classifying diagram:
\[ 
\xymatrix{  
Z \times_K \C^d \ar[r] \ar[d]  & EK \times_K \C^d \ar[d] \\ 
W_C \ar[r]^{f}  & BK  
} 
\]
One can also think of this vector bundle as the quotient by $K$ of the trivial $K$-equivariant
vector bundle $Z\times \C^d \to Z$ that one gets by restricting the tangent bundle of $\C^d$ to
$Z$. Let  $\fk_\C$ denote the complexified Lie algebra of $K$. The trivial vector bundle
$W_C \times \fk_\C  \to W_C$ can be seen as a sub-bundle of $Z \times_K \C^d \to W_C$
via the map
\begin{align}
W_C \times \fk_\C & \longrightarrow Z \times_K \C^d \notag \\
([z], v ) \  & \longmapsto\  [z, X_v] \notag
\end{align}
where we use the description of $W_C$ as $Z/K$ and $X_v \in T_z \C^d \cong \C^d$ is induced
by the free action of $K$ on $Z$. The quotient bundle $(Z \times_K \C^d)/ (W_C \times \fk_\C)$ is naturally
isomorphic to $T W_C$ and this shows that
\[
Z \times_K \C^d \cong TW_C \oplus (W_C \times \fk_\C)\,.
\]
Hence
\[
c_1 (TW_C) = c_1 (Z \times_K \C^d) = f^\ast c_1 (EK \times_K \C^d)
\quad\text{and}\quad
c_1 (TW_C) = 0 \ \Leftrightarrow\ c_1 (EK \times_K \C^d) = 0\,.
\]
Since
\[
H^2 (BK;\Z) \cong \text{character group of $K$}
\]
and $c_1 (EK \times_K \C^d) \in H^2 (BK;\Z)$ is given by
\[
c_1 (EK \times_K \C^d) = \chi_1 + \ldots + \chi_d\,,
\]
the result follows.
\end{proof}
\begin{remark} \label{rmk:c_1}
Let $k_1, \ldots, k_{d-n-1} \in \Z^d \subset\R^d$ be an integral basis
for the Lie algebra of $K\subset\T^d$. Proposition~\ref{prop:c_1}
states that
\[
c_1 (TW_C) = 0 \ \Leftrightarrow\ \sum_{j=1}^d (k_i)_j = 0 \,,\ 
\forall \, i=1, \ldots,d-n-1\,.
\]
\end{remark}

\subsection{Sasaki contact forms and Reeb vectors}

Let $(W,\omega, X)$ be a good toric symplectic cone of dimension $2(n+1)$, with
corresponding closed toric manifold $(N,\xi)$. Denote by $\Xx_X (W, \omega)$
the set of $X$-preserving symplectic vector fields on $W$ and by $\Xx (N,\xi)$
the corresponding set of contact vector fields on $N$. The $\T^{n+1}$-action
associates to every vector $\nu \in \ft \cong \R^{n+1}$ a vector field
$R_\nu \in \Xx_X (W,\omega) \cong \Xx (N, \xi)$.

\begin{defn} \label{def:sasaki}
A contact form $\alpha \in \Omega^1 (N,\xi)$ is called \emph{Sasaki} if
its Reeb vector field $R_\alpha$ satisfies
\[
R_\alpha = R_\nu \quad\text{for some $\nu\in\R^{n+1}$.}
\]
In this case we will say that $\nu\in\R^{n+1}$ is a \emph{Reeb vector}.
\end{defn}

In the context of their work on toric Sasaki geometry, Martelli-Sparks-Yau
characterize in~\cite{MSY} which $\nu\in\R^{n+1}$ are Reeb vectors of a
Sasaki contact form on $(N,\xi)$.

\begin{prop} \label{prop:sasaki} (\cite{MSY} )
Let $\nu_1, \ldots, \nu_d \in \R^{n+1}$ be the defining integral normals
of the moment cone $C\in\R^{n+1}$ associated with $(W,\omega,X)$ and 
$(N,\xi)$. The vector field $R_\nu \in \Xx_X (W,\omega) \cong \Xx(N,\xi)$
is the Reeb vector field of a Sasaki contact form 
$\alpha_\nu \in \Omega^1 (N,\xi)$ if and only if
\[
\nu = \sum_{j=1}^d a_j \nu_j \quad\text{with $a_j\in\R^+$ for all
$j=1, \ldots, d$.}
\]
\end{prop}

\begin{proof} (Outline)

This result is well-known for $(\C^d\setminus\{0\}, \om_{\rm st}, X_{\rm st})$.
In fact, any such Reeb vector field $R_\nu$ corresponds to
\[
\nu = \sum_{j=1}^d a_j e_j = (a_1, \ldots, a_d) \in (\R^+)^d
\]
and can be written in complex coordinates as
\[
R_\nu = i \sum_{j=1}^d a_j (z_j \frac{\partial}{\partial z_j} - 
\oz_j \frac{\partial}{\partial \oz_j})\,.
\]
The corresponding Reeb flow is given by
\[
\left(R_\nu\right)_s \cdot (z_1, \ldots, z_d) =
(e^{i a_1 s} z_1, \ldots, e^{i a_d s} z_d)
\]
and the contact form $\alpha_\nu$ is the restriction of
\[
\alpha_{\rm st} := \iota(X_{\rm st}) \om_{\rm st} = 
\frac{i}{2} \sum_{j=1}^d (z_j d \oz_j - \oz_j d z_j)
\]
to
\begin{align}
S^{2d-1} & \cong \{z\in\C^d\,:\ (\alpha_{\rm st})_z (R_\nu) = 1 \} \notag \\
& = \{z\in\C^d\,:\ \sum_{j=1}^d a_j |z_j|^2 = 1 \}\,. \notag
\end{align}
(Compare with the example in the Introduction and Example~\ref{ex:R}.)

The result follows for any good toric symplectic cone $(W, \omega, X)$, with
moment cone $C$, from the explicit reduction construction of the model
$(W_C, \omega_C, X_C)$. Note in particular the definition of the linear
map $\beta:\R^d \to \R^{n+1}$ given by~(\ref{def:beta}).
\end{proof}

\section{The Conley-Zehnder index}
\label{index}

\subsection{The Maslov index for loops of symplectic matrices}

Let $Sp(2n)$ denote the symplectic linear group, i.e. the group of linear
transformations of $\R^{2n}$ that preserve its standard linear symplectic
form. The Maslov index provides an explicit isomorphism 
$\pi_1 (Sp(2n))\cong \Z$. It assigns an integer $\mu_M (\varphi)$ to every loop
$\varphi:S^1= \R/2\pi\Z \to Sp(2n)$, uniquely characterized by the following properties:
\begin{itemize}
\item Homotopy: two loops in $Sp(2n)$ are homotopic iff they have the same
Maslov index.
\item Product: for any two loops $\varphi_1, \varphi_2 : S^1 \to Sp(2n)$ we
have
\[
\mu_M (\varphi_1 \cdot \varphi_2) = \mu_M (\varphi_1) + \mu_M (\varphi_2)\,.
\]
In particular, the constant identity loop has Maslov index zero.
\item Direct Sum: if $n=n_1 + n_2$, we may regard $Sp(2n_1)\oplus Sp(2n_2)$ as 
a subgroup of $Sp(2n)$ and
\[
\mu_M (\varphi_1 \oplus \varphi_2) = \mu_M (\varphi_1) + \mu_M (\varphi_2)\,.
\]
\item Normalization: the loop $\varphi : S^1 \to U(1) \subset Sp(2)$ defined
by $\varphi (\theta) = e^{i\theta}$ has Maslov index one.
\end{itemize}

\subsection{The Conley-Zehnder index for paths of symplectic matrices}

Robin and Salamon~\cite{RS} defined a Conley-Zehnder index which assigns
a half-integer $\mu_{CZ} (\Gamma)$ to any path of symplectic matrices
$\Gamma:[a,b]\to Sp(2n)$. This Conley-Zehnder index satisfies the following properties:
\begin{itemize}
\item[1)] Naturality: $\mu_{CZ} (\Gamma) = \mu_{CZ} (\psi \Gamma \psi^{-1})$
for all $\psi\in Sp(2n)$.
\item[2)] Homotopy: $\mu_{CZ} (\Gamma)$ is invariant under homotopies of
$\Gamma$ with fixed endpoints.
\item[3)] Zero: if $\Gamma(a)$ is the identity matrix and $\Gamma (t)$ has
no eigenvalue on the unit circle for $t\in (a,b]$, then $\mu_{CZ}(\Gamma) = 0$.
\item[4)] Direct Sum: if $n=n_1 + n_2$, we may regard $Sp(2n_1)\oplus Sp(2n_2)$ as 
a subgroup of $Sp(2n)$ and
\[
\mu_{CZ} (\Gamma_1 \oplus \Gamma_2) = \mu_{CZ} (\Gamma_1) + \mu_{CZ} (\Gamma_2)\,.
\]
\item[5)] Loop: if $\varphi:[a,b]\to Sp(2n)$ is a loop with $\varphi(a) = \varphi(b) =$
identity matrix, then
\[
\mu_{CZ} (\varphi \cdot \Gamma) = 2 \mu_M (\varphi) + \mu_{CZ} (\Gamma)\,.
\]
\item[6)] Concatenation: for any $a<c<b$ we have
\[
\mu_{CZ} (\Gamma) = \mu_{CZ}(\Gamma|_{[a,c]}) + \mu_{CZ}(\Gamma|_{[c,b]})\,.
\]
\item[7)] Signature: given a symmetric $(2n \times 2n)$-matrix $S$ with
$\|S\| < 1$, the Conley-Zehnder index of the path $\Gamma:[0,1]\to Sp(2n)$
defined by $\Gamma (t) = \exp (2\pi J_0 S t)$ is given by
\[
\mu_{CZ} (\Gamma) = \frac{1}{2} \sign{S}\,.
\]
Here $\|S\|:= \max_{|v|=1} |Sv|$, using the standard Euclidean norm on $\R^{2n}$,
$\sign(S) :=$ signature of the matrix $S$, i.e. the number of positive minus the
number of negative eigenvalues, and $J_0$ is the matrix representing the standard
complex structure on $\R^{2n}$, i.e
\[
J_0 = 
\begin{bmatrix}
0 & -I \\
I & 0
\end{bmatrix}\,.
\]
\item[8)] Shear axiom: the index of a symplectic shear
\[
\Gamma(t) = \left(\begin{matrix}
I & B(t) \\
0 & I \\
\end{matrix}\right)\,
\]
is given by $\frac 12\sign{B(a)} - \frac 12\sign{B(b)}$.
\end{itemize}

\begin{example} \label{ex:CZ1}
If $\Gamma:[a,b]\to Sp(2n)$ is a loop then
\[
\mu_{CZ} (\Gamma) = 2 \mu_M (\Gamma)\,.
\]
\end{example}

\begin{example} \label{ex:CZ2}
Let $T>0$ and $\Gamma : [0,T]\to U(1) \subset Sp(2)$ be defined by
\[
\Gamma (t) = e^{2\pi i t} = 
\begin{bmatrix}
\cos (2\pi t) & -\sin (2\pi t) \\
\sin (2\pi t) & \cos (2\pi t)
\end{bmatrix} \,,
\]
Then
\[
\mu_{CZ} (\Gamma) = 
\begin{cases}
2T &\text{if $T\in\N$;} \\
2 \lfloor T \rfloor +1  &\text{otherwise.}
\end{cases}\,, 
\]
where $\lfloor T \rfloor := \max\{n \in \Z;\, n \leq T\}$. 
\end{example}

\subsection{The Conley-Zehnder index for contractible periodic Reeb orbits}

We will now define the Conley-Zehnder index of a periodic Reeb orbit which, 
for the sake of simplicity, we will assume to be contractible.

Let $(N^{2n+1},\xi)$ be a co-oriented contact manifold, with 
contact form $\alpha$ and Reeb vector field $R_\alpha$.
Given a contractible periodic Reeb orbit $\gamma$, consider a 
\emph{capping disk} of $\gamma$, that is a map $\sigma_\gamma: D \to N$ that satisfies
\[
\sigma_\gamma|_{\partial D} = \gamma. 
\] 
Choose a (unique up to homotopy) symplectic trivialization
\[
\Phi: \sigma_\gamma^*\xi \to D \times \R^{2n} \,.
\]
We can define the symplectic path
\begin{equation}
\label{spath}
\Gamma(t) = \Phi(\gamma(t)) \circ d(R_\alpha)_t(\gamma(0))|_\xi \circ 
\Phi^{-1}(\gamma(0)).
\end{equation}
The Conley-Zehnder index of $\gamma$ with respect to the capping disk
$\sigma_\gamma$ is defined by
\[
\mu_{CZ}(\gamma,\sigma_\gamma) = \mu_{CZ}(\Gamma)\,.
\]
This is in general a half integer number and it is an integer number if the periodic 
orbit is nondegenerate. This means that the linearized Poincar\'e map of 
$\gamma$ has no eigenvalue equal to one.

This index in general does depend on the choice of the capping disk. 
More precisely, given another capping disk $\bar\sigma_\gamma$, we
have that
\[
\mu_{CZ}(\gamma,\bar\sigma_\gamma) - \mu_{CZ}(\gamma,\sigma_\gamma) = 
2\langle c_1(\xi), \bar\sigma_\gamma \# (-\sigma_\gamma)\rangle\,,
\]
where $ \bar\sigma_\gamma \# (-\sigma_\gamma)$ denotes the homology class of the gluing of the capping disks $\bar\sigma_\gamma$ and $\sigma_\gamma$ with the reversed orientation. Notice however that the parity of the index of a nondegenerate closed orbit 
does not depend on the chosen capping disk. In particular, the index of a 
contractible nondegenerate periodic orbit is a well defined element in 
$\Z/2c(\xi)\Z$, where
\[
c(\xi) := \inf \{k>0; \exists A \in \pi_2(N), \langle c_1(\xi),A\rangle = k\}
\]
is the {\it minimal Chern number} of $\xi$ (here we adopt the convention that 
the infimum over the empty set equals $\infty$).

\begin{rem}
We can define the Conley-Zehnder index of a contractible periodic orbit $\gamma$ of a Hamiltonian 
flow on a symplectic manifold $V$ in the same way, taking a capping disk $\sigma_\gamma$ and a trivialization of $TV$ over $\sigma_\gamma$. Analogously to periodic orbits of Reeb flows, the difference of the indexes with respect to two capping disks $\bar\sigma_\gamma$ and $\sigma_\gamma$  is given by
\begin{equation}
\label{changecap}
\mu_{CZ}(\gamma,\bar\sigma_\gamma) - \mu_{CZ}(\gamma,\sigma_\gamma) = 
2\langle c_1(TV), \bar\sigma_\gamma \# (-\sigma_\gamma)\rangle\,,
\end{equation}
where $c_1(TV)$ is the first Chern class of $TV$.
\end{rem}

\subsection{Behavior of the Conley-Zehnder index under symplectic reduction}

In this section we address the question of the relation between the Conley-Zehnder 
index of a periodic orbit and the Conley-Zehnder index of its symplectic reduction. 
This will be important later. Again, for the sake of simplicity, we will only 
consider contractible periodic orbits.

Let $V$ be a symplectic manifold and $h: V \times \R \to \R$ a time-dependent 
Hamiltonian on $V$ with a first integral $f: V \to \R$, that is, a function $f$ constant along the orbits of $h$. Denote by $X^t_h$ and $X_f$ 
the Hamiltonian vector fields of $h$ and $f$ respectively. Consider a Riemannian 
metric on $V$ induced by a compatible almost complex structure.

Let $Z$ be a regular level of $f$ and suppose that $X_f$ generates a free circle action on $Z$. Denote by $W$ the Marsden-Weinstein reduced symplectic manifold $Z/S^1$. The Hamiltonian $h$ induces a Hamiltonian $g$ on $W$ whose Hamiltonian flow 
$\psi_t$ satisfies the relation $\pi \circ \vr_t = \psi_t \circ \pi$, where $\vr_t$ 
is the Hamiltonian flow of $h$ and $\pi: Z \to W$ is the quotient projection. 
In particular, every periodic orbit $\tgamma$ of $X^t_h$ gives rise to a 
periodic orbit $\gamma = \pi \circ \tgamma$ of $X^t_g$ with the same period.

\begin{lem} \label{lem:main}
Suppose that the linearized Hamiltonian flow of $h$ on $Z$ leaves the distribution $\text{span}\{\nabla f\}$ invariant. Let $\tgamma$ be a closed orbit of $X_h$ 
contractible in $Z$ and $\sigma_{\tgamma}: D \to Z$ a capping disk for 
$\tgamma$. Then the capping disk $\sigma_\gamma := \pi \circ \sigma_{\tgamma}$ for the reduced 
periodic orbit $\gamma$ satisfies
\[
\mu_{CZ} (\tgamma,\sigma_{\tgamma}) = \mu_{CZ} (\gamma,\sigma_\gamma)\,.
\]
\end{lem}

\begin{proof}
Denote by $\Dd$ the symplectic distribution generated by $X_f$ and $\nabla f$. 
The hypothesis on $\nabla f$ and the fact that $f$ is a first integral for $h$ imply 
that $\Dd$ is invariant under $d\vr_t$. Hence, the symplectic orthogonal complement 
$\Dd^\omega$ is also invariant under $d\vr_t$. 

Let $\Phi: \sigma_\tgamma^* TV \to D^2 \times \R^{2d}$ be a (unique up to homotopy)
trivialization of $TV$ over $\sigma_\tgamma$ ($\dim V = 2d$). Since $X_f$, $\nabla f$ 
and $\Dd^\omega$ are defined over the whole disk $\sigma_\tgamma$, one can find a
symplectic bundle isomorphism $\Psi: D^2 \times \R^{2d} \to D^2 \times \R^{2d}$ that covers the identity and satisfies:
\begin{itemize}
\item[(P1)] $\pi_2(\Psi(\Phi(X_f))) = e_1$ and $\pi_2(\Psi(\Phi(\nabla f))) \in \text{span}\{f_1\}$, where $\{e_1,\ldots,e_d,f_1,\ldots,f_d\}$ is a fixed symplectic
basis in $\R^{2d}$ and $\pi_2: D^2 \times \R^{2d} \to \R^{2d}$ is the projection onto 
the second factor;
\item[(P2)] $\pi_2(\Psi(\Phi(\sigma_\tgamma^* \Dd^\omega))) = \text{span}\{e_2,\ldots,e_n,f_2,\ldots,f_n\}$.
\end{itemize}

Note that $\sigma_\gamma = \pi \circ \sigma_\tgamma$ is a capping disk for 
$\gamma$ and the differential of $\pi$ induces the identification 
$d\pi|_{\Dd^\omega}: \sigma_\tgamma^* \Dd^\omega \to \sigma_\gamma^* TW$. Hence, 
in order that $\Phi$ induces a trivialization over $\sigma_\gamma$ it is enough 
to choose it such that it sends $\Dd^\omega$ to a fixed symplectic subspace in 
$\R^{2d}$. Property (P2) ensures that the trivialization 
$\Lambda := \Psi \circ \Phi$ satisfies this property. In fact, consider the 
splitting $\R^{2d} = E_1 \oplus E_2$, where $E_1 = \text{span}\{e_1,f_1\}$ and 
$E_2 = \text{span}\{e_2,\ldots,e_n,f_2,\ldots,f_n\}$. We have that $\Lambda(\sigma_\tgamma^* \Dd^\omega) = D^2 \times E_2$ and the trivialization 
over $\sigma_\gamma$ is then given by
\[
\Lambda \circ (d\pi|_{\Dd^\omega})^{-1}: \sigma_\gamma^* TW \to D^2 \times E_2\,.
\]
Now, define the symplectic path
\[
\Gamma(t) = \Lambda(\tgamma(t))\circ d\vr_t(\tgamma(0))\circ 
\Lambda^{-1}(\tgamma(0))\,,
\]
so that $\mu(\tgamma,\sigma_\tgamma) = \mu(\Gamma)$. Since $f$ is a first integral, 
$X_f$ is preserved by $d\vr_t$ and, by hypothesis, $\text{span}\{\nabla f\}$ is 
preserved as well. Thus, by property (P1), $\Gamma|_{E_1}$ is a symmetric symplectic path in $\R^2$ with an eigenvalue one. But a symmetric symplectic isomorphism in $\R^2$ with an eigenvalue one is necessarily the identity.

Consequently, the Direct Sum property of the index yields
\[
\mu(\Gamma) = \mu(\Gamma|_{E_1}) + \mu(\Gamma|_{E_2}) = \mu(\Gamma|_{E_2}) = \mu(\gamma,\sigma_\gamma)\,, 
\]
finishing the proof of the Lemma.
\end{proof}

\begin{rem}
The assumption on $\text{span}\{\nabla f\}$ is necessary. In order to show this, 
consider the Hamiltonian $h: \C^2 \to \R$ given by $h(z_1,z_2) = g(|z_1|^2+|z_2|^2)$,
where $g$ is a smooth real function. It is obviously invariant under the Hamiltonian 
circle action generated by $f(z_1,z_2) = |z_1|^2 + |z_2|^2$ whose reduced symplectic 
manifold is $S^2$. Every reduced orbit is a constant solution whose constant capping 
disk has index zero. Consequently, by equation \eqref{changecap} and the fact that $c_1(TS^2) = 2$, the index of a reduced orbit is given by an integer multiple of four, whatever is the choice of the capping disk. However, one can show that a nonconstant orbit $\tgamma$ of $h$ has 
index
\begin{equation*}
\mu(\tgamma) =
\begin{cases}
7/2\text{ if }g^{\prime\prime}(f(\tgamma))<0\\
4\text{ if }g^{\prime\prime}(f(\tgamma))=0\\
9/2\text{ if }g^{\prime\prime}(f(\tgamma))>0
\end{cases}.
\end{equation*}
The hint to show this is the existence of a trivialization over a capping disk
$\sigma_\tgamma$ such that the linearized Hamiltonian flow restricted to the 
subspace spanned by $X_f$ and $\nabla f$ is given by the symplectic shear
\[
\left(\begin{matrix}
1 & -g^{\prime\prime}(f(\tgamma))t \\
0 & 1 \\
\end{matrix}\right)\,.
\]
Note that the linearized Hamiltonian flow of $h$ preserves $\nabla f$ precisely 
when $g^{\prime\prime}(f(\tgamma))=0$.
\end{rem}

\begin{rem}
The hypothesis that $\tgamma$ is contractible \emph{in} $Z$ is also necessary (notice that to define the Conley-Zehnder index of $\tgamma$ we need only to suppose that $\tgamma$ is contractible \emph{in} $V$). 
As a matter of fact, let $W$ be a symplectic manifold with first Chern class 
different from one and consider on $V := W \times \C$ the circle action generated 
by $f(p,z) = |z|^2$. It is easy to see that every orbit $\tgamma$ of $f$ has a 
capping disk with index two. On the other hand, the reduced orbit is a constant 
solution $\gamma$ whose constant capping disk has index zero. Consequently, the 
hypothesis on $c_1(TW)$ and equation \eqref{changecap} imply that there is no capping disk for $\gamma$ with 
index two. A less trivial argument can give examples where the orbits are not 
contractible but homologically trivial.
\end{rem}

\section{Cylindrical contact homology}
\label{s:cch}

There are several versions of contact homology (see~\cite{Bo3} for a survey). 
A suitable one for our purposes is cylindrical contact homology whose definition 
is closer to the usual construction of Floer homology \cite{Fl1, Fl2, Fl3} but 
with some rather technical differences. The aim of this section is to sketch 
this construction. Details can be found in \cite{Bo1,EGH,Ust,vK} and references 
therein.

Let $\alpha$ be a contact form on $N^{2n+1}$ with contact structure 
$\xi = \text{ker }\alpha$ and let $R_\alpha$ be its Reeb vector field. 
For the sake of simplicity, we will assume that $c_1(\xi) = 0$. Denote by $\Pp$ 
the set of periodic orbits of $R_\alpha$ and suppose that $R_\alpha$ is 
nondegenerate, i.e. every closed orbit $\gamma \in \Pp$ is nondegenerate. 
A periodic orbit of $R_\alpha$ is called {\it bad} if it is an even multiple 
of a periodic orbit whose parities of the Conley-Zehnder index of odd and even 
iterates disagree. An orbit that is not bad is called {\it good}. Denote the set 
of good periodic orbits by $\Pp^0(\alpha)$.

Consider the chain complex $CC_*(\alpha)$ given by the graded group with 
coefficients in $\Q$ generated by good periodic orbits of $R_\alpha$ graded 
by their Conley-Zehnder index plus $n-2$. This extra term $n-2$ is not important 
in the definition of cylindrical contact homology but the reason for its use 
will be apparent later. Let us denote the degree of a periodic orbit by $|\gamma|$.

The boundary operator $\partial$ is given by counting rigid holomorphic cylinders 
in the symplectization $(W,\om) := (\R \times N,d(e^t\alpha))$. More precisely, 
fix an almost complex structure $J$ on $W$ compatible with $\om$ such that $J$ is invariant by $t$-translations,  
$J(\xi) = \xi$ and $J(\frac{\partial}{\partial t}) = R_\alpha$. The space of 
these almost complex structures is contractible. Let 
$\Sigma = S^2\setminus\Gamma$ be a punctured rational curve, where $S^2$ is endowed with a 
complex structure $j$ and $\Gamma = \{x,y_1,...,y_s\}$ is the set of 
(ordered) punctures of $\Sigma$. We will consider holomorphic curves from $\Sigma$ 
to the symplectization $W$, that is, smooth maps $F = (a,f): \Sigma \to W$ 
satisfying $dF \circ j = J \circ dF$. We restrict ourselves to holomorphic 
curves such that, for polar coordinates $(\rho,\theta)$ centered at a puncture 
$p \in \Gamma$, the following conditions hold:
\begin{equation*}
\lim_{\rho\to 0} a(\rho,\theta) = 
\begin{cases}
+\infty \text{ if $p=x$} \\
-\infty \text{ if $p=y_i$ for some $i=1,\ldots,s$}.
\end{cases}
\end{equation*}
\begin{equation*}
\lim_{\rho\to 0} f(\rho,\theta) = 
\begin{cases}
\gamma(-T\theta/2\pi) \text{ if $p=x$} \\
\gamma_i(T_i\theta/2\pi) \text{ if $p=y_i$ for some $i=1,\ldots,s$}.
\end{cases}
\end{equation*}
where $\gamma$ and $\gamma_i$ are good periodic orbits of $R_\alpha$ of periods 
$T$ and $T_i$ respectively.
Denote the set of such holomorphic curves by $\Mm(\gamma,\gamma_1,...,\gamma_s;J)$ and notice that $j$ is not fixed. Define an equivalence relation $\simeq$ on $\Mm(\gamma,\gamma_1,...,\gamma_s;J)$ 
by saying that $(F = (a,f),j)$ and $(\tilde F = (\tilde a,\tilde f),\tilde j)$ are equivalent 
if there is a shift $\tau \in \R$ and a biholomorphism $\vr: (S^2,j) \to 
(S^2,\tilde j)$ 
such that $\vr(p)=p$ for every $p \in \Gamma$ and
\[
(a,f) = (\tilde a \circ \vr + \tau,\tilde f \circ \vr)\,.
\]
Define the moduli space $\widehat\Mm(\gamma,\gamma_1,...,\gamma_s;J)$ as $\Mm(\gamma,\gamma_1,...,\gamma_s;J)/\simeq$. A crucial ingredient in order to understand the set $\widehat\Mm(\gamma,\gamma_1,...,\gamma_s;J)$ is the operator 
$D_{(F,j)}: T_F\mathcal B^{1,p,\delta}(\Sigma,V) \times T_j\mathcal T \to L^{p,\delta}(\Sigma,F^*TV)$ called the 
{\it vertical differential} and given by
\[
D_{(F,j)} (\psi,y) = \nabla\psi + J\circ\nabla\psi\circ j + (\nabla_\psi J)\circ DF\circ j + J\circ DF\circ y,
\]
where $p>2$, $\delta>0$ is sufficiently small, $\mathcal B^{1,p,\delta}(\Sigma,V)$ is the Banach manifold consisting of $W^{1,p}_{\text{loc}}$ maps from $\Sigma$ to $W$ with a suitable behavior near the punctures, $\mathcal T$ stands for a Teichm\"uller slice through $j$ as defined in \cite{Wendl} and $L^{p,\delta}(\Sigma,F^*TV)$ is a weighted Sobolev space given by the completion of the space of smooth anti-holomorphic 1-forms $\Omega^{0,1}(\Sigma,F^*TV)$ with respect to suitable norms, see \cite{Wendl} for details. Notice that we are tacitly taking the Levi-Civita connection given by the metric induced by the symplectic form and the almost complex structure.

This is a Fredholm operator with index given by
\[
|\gamma| - \sum_{i=1}^s |\gamma_i| + \dim \text{Aut}(\Sigma,j),
\]
where $\text{Aut}(\Sigma,j)$ is the group of automorphisms of $(\Sigma,j)$, see page 376 in \cite{Wendl}.
We say that $J$ is {\it regular} if $D_{(F,j)}$ is surjective for every holomorphic 
curve $(F,j)$ (it does not depend on the choice of the Teichm\"uller slice, see Lemma 3.11 in \cite{Wendl}). It turns out that if $J$ is regular then 
$\widehat\Mm(\gamma,\gamma_1,...,\gamma_s;J)$ is a smooth manifold with dimension given 
by
$$ |\gamma| - \sum_{i=1}^s |\gamma_i| -1. $$
Moreover, $\widehat\Mm(\gamma,\gamma_1,...,\gamma_s;J)$ admits a compactification 
$\overline \Mm(\gamma,\gamma_1,...,\gamma_s;J)$ with a coherent orientation \cite{BM} whose boundary is given by 
holomorphic buildings \cite{BEH+}. In particular, if $J$ is regular, then 
$\overline \Mm (\gamma,\gamma_1,...,\gamma_s;J)$ is a finite set with signs whenever $|\gamma| - \sum_{i=1}^s|\gamma_i| = 1$. 

However, unlike Floer homology in the monotone case, regularity is not achieved 
in general by a generic choice of $J$. Instead, one needs to use multi-valued
perturbations equivariant with respect to the action of biholomorphisms and this turns out 
to be a very delicate issue. Several ongoing approaches have been developed to 
give a rigorous treatment to this problem, see \cite{CM,HWZ1,HWZ2,HWZ3}.
Consequently, following \cite[Remark 9]{BO}, we will assume the following technical condition throughout this work.

\vskip .2cm
\noindent {\bf Transversality assumption.}
We suppose that the almost complex structure $J$ is regular for holomorphic curves with index less or equal than two (the index of a holomorphic curve is defined as the degree of the positive periodic orbit minus the sum of the degrees of the negative ones). Moreover, we will also assume the existence of regular almost complex structures for holomorphic curves with index less or equal than one in cobordisms and with index less or equal than zero in 1-parameter families of cobordisms.
\vskip .2cm

We have then the following result on the structure of moduli spaces of holomorphic cylinders in symplectizations.

\begin{prop}\cite{EGH}
\label{transversality}
Under the previous transversality assumption, the moduli spaces \linebreak $\overline \Mm(\gamma,\gamma_1)$ of dimension zero consist of finitely many points with rational weights. The moduli spaces $\overline \Mm(\gamma,\gamma_1)$ of dimension one have boundary given by finitely many points corresponding to holomorphic buildings with rational weights whose sum counted with orientations vanishes. Moreover, if a holomorphic building in the boundary consists of a broken cylinder then its weight is given by the product of the weights of each cylinder.
\end{prop}

We expect the transversality assumption to be completely removed using the polyfold theory developed by Hofer, Wysocki and Zehnder, see \cite{HWZ1,HWZ2,HWZ3}.

Thus fix $s=1$, that is, let $\Sigma$ be a cylinder. By the discussion above, 
if two periodic orbits $\gamma$ and $\bar\gamma$ satisfy 
$|\gamma| = |\bar\gamma|+1$ then $\overline\Mm(\gamma,\bar\gamma)$ is a finite set. 
This enables us to define the boundary operator in the following way. 
Let $\gamma$ be a periodic orbit of multiplicity $m(\gamma)$, i.e. $\gamma$ is a 
covering of degree $m(\gamma)$ of a simple closed orbit. Define
\[
\partial\gamma = m(\gamma)\sum_{\bar\gamma \in \Pp^0(\alpha),|\bar\gamma| = |\gamma|-1} \sum_{F \in \overline\Mm(\gamma,\bar\gamma)}\text{sign}(F)\text{weight}(F) \bar\gamma\,,
\]
where $\text{sign}(F)$ is the sign of $F$ determined by the coherent orientation of $\overline\Mm(\gamma,\bar\gamma)$ and $\text{weight}(F)$ is the weight established in the previous proposition. A somewhat different definition of the boundary operator is given in \cite{EGH} using asymptotic markers, but one can check that this is equivalent to the definition above. Notice the similarity with Floer homology, but we have to consider weights in the boundary operator.

The next proposition is a generalization of section 1.9.2 in \cite{EGH}, where it is shown that cylindrical contact homology 
is well defined and an invariant of the contact structure for nice contact forms. The specific nature of the weights in the boundary operator does not play any role in the proof; the point is to avoid the presence of certain tree-like curves in the boundary of moduli spaces of dimension one and it is here that the hypothesis on the contact forms comes in. As a matter of fact, as will be accounted in the proof, the assumption that the contact form is even implies that there is no holomorphic curve of index one in the symplectization and the hypothesis of non-existence of periodic orbits of degree $1$, $0$ and $-1$ is to avoid rigid planes (rigid means that it belongs to a moduli space of dimension zero) in symplectizations, cobordisms and cobordisms in 1-parameter families of cobordisms respectively. 

Following exactly as in the proof in \cite{Bo1,EGH,Ust}, one can extend the argument to even contact forms and prove 
Proposition~\ref{prop:even-1}, which we restate here for the convenience of the 
reader.

\begin{prop} \label{prop:even-2}
Let $(N,\xi)$ be a contact manifold with an even or nice nondegenerate 
contact form $\alpha$. Then the boundary operator 
$\partial:C_\ast (N,\alpha) \to C_{\ast-1} (N,\alpha)$ satisfies $\partial^2 = 0$ 
and the homology of $(C_\ast (N,\alpha), \partial)$ is independent of the choice 
of even or nice nondegenerate contact form $\alpha$. Hence, the cylindrical 
contact homology $HC_\ast(N,\xi; \Q)$ is a well defined invariant of the contact
manifold $(N,\xi)$.
\end{prop}

\begin{proof}
The proof follows the proofs in \cite{Bo1,EGH,Ust}. We will just 
recall the main steps and explain how to proceed with even contact forms.

If $\alpha$ is even then obviously $\partial^2=0$, since $\partial=0$. To deal with the case that 
$\alpha$ is nice, notice that $\partial^2$ counts broken rigid holomorphic 
cylinders in the symplectization of $\alpha$ that appear (by a gluing argument) as points in the 
boundary of the moduli space of cylinders connecting orbits with index difference 
equal to two. The condition that $R_\alpha$ has no periodic orbit of degree 
one implies that there is nothing else in the boundary of this moduli space. Indeed, we could have in the boundary a tree-like curve with one level of index 1 and the other level consisting of a rigid plane and a vertical cylinder. However, the non-existence of orbits of degree one excludes the existence of these planes.

Thus, $\partial^2$ counts points in the boundary of the moduli space of dimension one and the sum of the weights of these points counted with orientations vanishes. These weights are the products of the weights of the rigid holomorphic cylinders in each level and the number of ways that such cylinders can be glued to each other is given precisely by the multiplicity of the closed orbit where we glue. This is the reason why the factor $m(\gamma)$ appears in the definition of $\partial$. Hence it follows that $\partial^2=0$. 

Now, let us consider the invariance problem. To carry it out, we will construct 
an isomorphism $\Phi: HC_*(N,\alpha) \to HC_*(N,\tilde\alpha)$. Since 
$\tilde\alpha$ defines the same contact structure as $\alpha$ we can write 
$\tilde\alpha = f\alpha$, where $f: N \to \R$ is a smooth positive function. 
Take a function $g: \R \times N \to \R$ such that $g(t,x) = e^t$ for $t>R$, 
$g(t,x) = e^t f(x)$ for $t<-R$ and $\partial_t g>0$, where $R>0$ is a constant big enough. 
It is easy to check that $d(g\alpha)$ is a symplectic form on $\R \times N$. We call 
$(W,\om) := (\R \times N,d(g\alpha))$ a symplectic cobordism with ends 
$W_+ = (R,+\infty) \times N$ and $W_- = (-\infty,-R) \times N$, restricted to 
which $\om$ coincides with the symplectic forms of the symplectizations of 
$\alpha$ and $\tilde\alpha$ respectively. Denote by $J_+$ and $J_-$ the 
corresponding almost complex structures on $W_+$ and $W_-$ as defined previously 
and consider a compatible almost complex structure $J_W$ on $W$ that extends 
$J_-$ and $J_+$.

In order to define $\Phi$, we need to consider holomorphic curves on $W$ in a 
similar fashion to what we did in symplectizations. More precisely, let
$\Sigma$ be as before and fix a periodic orbit 
$\gamma$ of $R_\alpha$ and periodic orbits $\tilde\gamma_1,...,\tilde\gamma_s$ 
of $R_{\tilde\alpha}$. We look at holomorphic curves 
$F: \Sigma \to W$ that are asymptotic to $\gamma$ and
$\tilde\gamma_1,...,\tilde\gamma_s$ at the positive and negative punctures 
respectively. Denote the set of such curves by $\Mm(\gamma,\tilde\gamma_1,...,\tilde\gamma_s;J_W)$.

Analogously to symplectizations, define an equivalence relation $\simeq$ on 
$\Mm(\gamma,\tilde\gamma_1,...,\tilde\gamma_s;J_W)$ by saying that $F = (a,f)$ 
and $\tilde F = (\tilde a,\tilde f)$ are equivalent if there is a biholomorphism 
$\vr: S^2 \to S^2$ that restricted to $\Gamma$ is the identity and
\[
(a,f) = (\tilde a \circ \vr,\tilde f \circ \vr)\,.
\]
The moduli space $\widehat\Mm(\gamma,\tilde\gamma_1,...,\tilde\gamma_s;J_W) 
:= \Mm(\gamma,\tilde\gamma_1,...,\tilde\gamma_s;J_W)/\simeq$ admits a 
compactification 
\[
\overline \Mm(\gamma,\tilde\gamma_1,...,\tilde\gamma_s;J_W)
\]
whose boundary is given by holomorphic buildings.

Under our transversality assumption, a result similar to Proposition \ref{transversality} holds for symplectic cobordisms 
and it establishes that one can choose $J_W$ such that the moduli spaces $\overline\Mm(\gamma,\tilde\gamma_1;J_W)$ of dimension zero and one have the desired properties. The dimension is given by $|\gamma| - |\gamma_1|$.

We define a map $\Psi: CC_*(N,\alpha) \to CC_*(N,\widetilde\alpha)$ by
\[
\Psi(\gamma) = m(\gamma)\sum_{\tilde\gamma \in \Pp^0(\tilde\alpha),|\tilde\gamma| 
= |\gamma|} \sum_{F \in \overline\Mm(\gamma,\tilde\gamma)}\text{sign}(F)\text{weight}(F)\tilde\gamma\,. 
\]
In order to show that $\Psi$ is a chain map, the idea, as in the proof of 
$\partial^2=0$, is to identify $\partial_{\tilde\alpha}\Psi - \Psi\partial_\alpha$ 
with the boundary of a a moduli space of dimension one. To achieve this identification consider the 
moduli space
\[
\overline\Mm(\gamma;J_W) := \bigcup_{\tilde\gamma \in \Pp^0(\tilde\alpha),\ |\gamma|= |\tilde\gamma| + 1} 
\overline\Mm(\gamma,\tilde\gamma;J_W)\,.
\]
By a gluing argument, the broken cylinders counted in 
$(\partial_{\tilde\alpha}\Psi - \Psi\partial_\alpha)(\gamma)$ 
are contained in $\partial\overline\Mm(\gamma;J_W)$. 
We need to show that there is nothing else than these broken cylinders in 
$\partial\overline\Mm(\gamma;J_W)$. But the compactness results show that the 
boundary is given by holomorphic buildings with two levels. Hence, we may have two possibilities:

\begin{itemize}
\item A pair of pants of index $0$ in the cobordism $W$ and a rigid plane and a vertical cylinder in the symplectization of $\tilde\alpha$.
\item A punctured sphere of index $1$ in the symplectization of $\alpha$ and (possibly several) rigid planes and a rigid cylinder in the cobordism $W$.
\end{itemize}

The first possibility does not hold because if $\tilde\alpha$ is even or nice then there is no rigid holomorphic plane in the symplectization of $\tilde\alpha$. The second possibility, in turn, is forbidden because, if $\alpha$ is even, there is no rigid holomorphic curve in the symplectization and, if $\alpha$ is nice, there is no rigid plane in the cobordism from $\alpha$ to $\tilde\alpha$, since there is no orbit of degree zero. This shows that $\Psi$ is a chain map and consequently it induces a map $\Phi$ in the homology. 

To prove that $\Phi$ is an isomorphism we construct its inverse. Consider 
the map $\tilde\Psi: CC_*(N,\tilde\alpha) \to CC_*(N,\alpha)$ obtained by 
the construction above switching $\alpha$ and $\tilde\alpha$. We claim that 
$\tilde\Psi \circ \Psi$ is chain homotopic to the identity. Indeed, we have 
that 
\[
\tilde\Psi \circ \Psi - Id = \partial_{\alpha} \circ A + 
A \circ \partial_{\alpha}\,,
\] 
where $A: CC_*(N,\alpha) \to CC_{*+1}(N,\alpha)$ is a map of degree $1$ 
obtained in the following way. 

Consider a 1-parameter family of symplectic cobordisms 
$W_\lambda:= (W,\omega_\lambda)$, $\lambda \in [0,1]$, such that 
$W_0$ is the symplectic cobordism given by the gluing of the cobordisms 
from $\alpha$ to $\tilde\alpha$ and from $\tilde\alpha$ to $\alpha$ and 
$W_1$ is the symplectization of $\alpha$. Let $J_\lambda$ be a smooth 
family of almost complex structures compatible with $\omega_\lambda$ and 
consider the set
\[
\overline\Mm(\gamma,\gamma_1,\ldots,\gamma_s;\{J_\lambda\}) = \{(\lambda,F);
\ 0 \leq \lambda \leq 1,\ F \in 
\overline\Mm(\gamma,\gamma_1,\ldots,\gamma_s;J_\lambda)\}\,.
\]
Once again, a result similar to Proposition~\ref{transversality} holds for 
1-parameter families of symplectic cobordisms establishing that one can choose 
$J_\lambda$ such that the moduli spaces $\overline\Mm(\gamma,\gamma_1;\{J_\lambda\})$ of dimension zero and one have the desired properties. Now, the dimension is given by $|\gamma| - |\gamma_1|+1$. Hence if $\gamma$ and $\bar\gamma$ are good periodic orbits of $R_\alpha$ such 
that $|\gamma|-|\bar\gamma| = -1$ then 
$\Mm(\gamma,\bar\gamma;\{J_\lambda\})$ is a finite set. Define
\[
A(\gamma) = m(\gamma)\sum_{\bar\gamma \in \Pp^0(\alpha),|\bar\gamma| = |\gamma|+1} \sum_{F \in \overline\Mm(\gamma,\bar\gamma;\{J_\lambda\})}\text{sign}(F)\text{weight}(F) \bar\gamma\,.
\]
By compactness results, if $|\gamma| = |\bar\gamma|$ then the boundary of 
$\overline\Mm(\gamma,\bar\gamma;\{J_\lambda\})$ is given by components coming 
from the boundary of $[0,1]$ and holomorphic buildings of height two. 
This first component is the union of $\overline\Mm(\gamma,\bar\gamma;J_0)$ 
and $\overline\Mm(\gamma,\bar\gamma;J_1)$ and it counts as 
$\tilde\Psi \circ \Psi - Id$, since every cylinder of index zero in the 
symplectization $W_1$ is trivial. The second one is given by broken cylinders 
of index zero counted by $\partial_{\alpha} \circ A + A \circ \partial_{\alpha}$ 
and, besides these broken cylinders, we might have three possibilities:

\begin{itemize}
\item A pair of pants of index $-1$ in a cobordism $W_\lambda$ and a rigid plane and a vertical cylinder in the symplectization of $\alpha$.
\item A pair of pants of index $1$ in the symplectization of $\alpha$ and a plane of index $-1$ and a cylinder of index $0$ in a cobordism $W_\lambda$.
\item A punctured sphere of index $1$ in the symplectization of $\alpha$ and (possibly several) planes of index $0$ and a cylinder of index $-1$ in a cobordism $W_\lambda$.
\end{itemize}

The first case is discarded because there is no rigid plane in the symplectization if $\alpha$ is even or nice. The second one does not hold if $\alpha$ is even since there is no rigid curve in the symplectization, and, if $\alpha$ is nice, there is no plane of index $-1$ in the cobordism. Finally, the third possibility cannot happen because, if $\alpha$ is even, there is no rigid curve in the symplectization and, if $\alpha$ is nice, there is no plane of index $0$ in the cobordism.
\end{proof}

\section{Proof of Theorem~\ref{thm:main1}}
\label{s:proof1}

Let us first describe the idea of the proof. We consider a Sasaki contact form on $N$ whose Reeb flow has finitely many nondegenerate simple periodic orbits $\gamma_\ell$, $\ell=1,...,m$, where $m$ is the number of edges of the good moment cone. Let $X$ be the Liouville vector field of the corresponding good symplectic cone and consider the $X$-invariant Hamiltonian flow associated to the Reeb flow. For each $\ell=1,...,m$ we choose a suitable lift of this Hamiltonian flow to a linear flow on $\R^{2d}$ using the symplectic reduction process described after Remark \ref{rem:orbBW}. More precisely, we require that the lift $\tilde\gamma_\ell$ of $\gamma_\ell$ is a {\it closed} orbit in $\R^{2d}$. This enables us to apply Lemma \ref{lem:main} and consequently reduces the proof to the computation of the Conley-Zehnder index of $\tilde\gamma_\ell$. For this computation, we can use the global trivialization of $T\R^{2d}$ and, since the lifted flow is given by a 1-parameter subgroup of $\T^d$ via the usual $\T^d$-action in $\R^{2d} \simeq \C^d$, the index is easily computed by the corresponding vector in the Lie algebra. This vector, given by equation (\ref{eq:vector_lift}), is completely determined by the associated good moment cone and it turns out that the degree of every orbit is an even number.

Let $(W, \omega, X)$ be a good toric symplectic cone determined 
by a good moment cone $C\subset(\R^{n+1})^\ast$ defined by
\[
C = \bigcap_{j=1}^d \{x\in(\R^{n+1})^\ast\,:\ 
\ell_j (x) := \langle x, \nu_j \rangle \geq 0\}\,
\]
where $d\geq n+1$ is the number of facets and each $\nu_j$ is a primitive 
element of the lattice $\Z^{n+1} \subset \R^{n+1}$ (the inward-pointing 
normal to the $j$-th facet of $C$).

Let $\nu\in\ft \cong \R^{n+1}$ be any vector in the Lie algebra
of the torus $\T^{n+1}$ satisfying the following two conditions:
\begin{itemize}
\item[(i)]
\[
\nu = \sum_{j=1}^d a_j \nu_j \quad\text{with $a_j\in\R^+$ for all
$j=1,\ldots,d$;}
\]
\item[(ii)] the $1$-parameter subgroup generated by $\nu$ is dense in
$\T^{n+1}$.
\end{itemize}
Let $R_\nu \in\Xx_X (W, \omega) \cong \Xx (N, \xi)$ be the Reeb vector
field of the Sasaki contact form $\alpha_\nu \in\Omega^1 (N, \xi)$.

\begin{lemma}
The Reeb vector field $R_\nu$ has exactly $m$ simple closed orbits, where
\[
m = \ \text{number of edges of $C$.}
\]
\end{lemma}
\begin{proof}
Under the moment map $\mu:W\to C\in(\R^{n+1})^\ast$, any $R_\nu$-orbit $\gamma$
is mapped to a single point $p\in C$. The pre-image $\mu^{-1}(p)$ is a 
$\T^{n+1}$-orbit and the fact that $\nu$ generates a dense $1$-parameter subgroup
of $\T^{n+1}$ implies that the $R_\nu$-orbit $\gamma$ is dense in $\mu^{-1}(p)$.
Hence, $\gamma$ is closed iff $\dim (\mu^{-1}(p)) = 1$, and this happens iff
$p$ belongs to a $1$-dimensional face of $P$, i.e. an edge.
\end{proof}
\begin{remark} \label{rmk:multiple}
If the toric symplectic cone $W$ is not simply connected, the simple closed Reeb
orbit $\gamma$ associated to an edge $E$ of the moment cone $C$ might not be
contractible. However, it follows from Proposition~\ref{prop:pi1} that a finite
multiple of $\gamma$ is contractible and ``simple closed Reeb orbit associated to
$E$'' will always mean ``smallest multiple of $\gamma$ that is contractible''.
\end{remark}

Let $E_1, \ldots, E_m$ denote the edges of $C$ and $\gamma_1, \ldots, \gamma_m$
the corresponding simple closed orbits of the Reeb vector field $R_\nu$. Since
$C$ is a good cone, each edge $E_\ell$ is the intersection of exactly $n$ facets
$F_{\ell_1}, \ldots, F_{\ell_n}$, whose set of normals
\[
\nu_{\ell_1}, \ldots, \nu_{\ell_n}
\]
can be completed to an integral base of $\Z^{n+1}$. Hence, for each
$\ell = 1, \ldots, m$, we can choose an integral vector $\eta_\ell \in\Z^{n+1}$
such that
\[
\left\{\nu_{\ell_1}, \ldots, \nu_{\ell_n}, \eta_\ell \right\} \quad
\text{is an integral base of $\Z^{n+1}$.}
\]

The map $\beta : \R^d \to \R^{n+1}$ defined by~(\ref{def:beta}) is surjective and
integral ($\beta (\Z^d) \subset \Z^{n+1}$). Hence, for each $\ell = 1, \ldots, m$, there is a 
smallest natural number $N_\ell \in \N$ and an integral vector $\teta_\ell \in \Z^d$ such that
\[
\beta (\teta_\ell) = N_\ell \,\eta_\ell\,.
\]
The Reeb vector field $R_\nu$ can be uniquely written as
\[
R_\nu = \sum_{i=1}^n b^\ell_i \nu_{\ell_i} + b^\ell N_\ell \eta_\ell\,,\ 
\text{with}\ b^\ell_1,\ldots, b^\ell_n, b^\ell \in \R\,,
\]
and we can then lift it to a vector $\tR_\nu^\ell \in \R^{d}$ as
\begin{equation}
\label{eq:vector_lift}
\tR_\nu^\ell = \sum_{i=1}^n b^\ell_i e_{\ell_i} + b^\ell \teta_\ell\,,
\end{equation}
so that
\[
\beta (\tR_\nu^\ell) = R_\nu\,.
\]

\begin{remark}
$N_\ell =$ ``smallest multiple'' considered in Remark~\ref{rmk:multiple}. If
the moment cone $C$ determines a simply connected toric symplectic cone $W$,
then $N_\ell = 1$, $\forall\,\ell=1,\ldots,m$.
\end{remark}

Recall from subsection~\ref{ss:models} that $W = Z / K$, where $K=\ker\beta\subset\T^d$
and
\[
Z=\phi_K^{-1}(0) \setminus\{0\} \equiv\ \mbox{zero level set of moment map in $\C^{d}\setminus\{0\}$.}
\]
The restriction to $Z\subset\C^d$ of the linear flow on $\C^d$ generated by 
$\tR_\nu^\ell$ is a lift of the Reeb flow on $W$ generated by $R_\nu$.
Consider
\begin{align}
Z & \longrightarrow W = Z/K \stackrel{\mu}{\longrightarrow} 
C \subset (\R^{n+1})^\ast \notag \\
z & \longmapsto [z] \notag
\end{align}
We have that 
\[
[z] \in \mu^{-1} (E_\ell) \Leftrightarrow z_{\ell_1} = \cdots = z_{\ell_n} = 0\,.
\]
This implies that $\gamma_\ell$ can be lifted to $Z$ as a closed orbit $\tgamma_\ell$
of $\tR_\nu^\ell$. The periods of $\gamma_\ell$ and $\tgamma_\ell$ are both given by
\[
T_\ell = \frac{2\pi}{b^\ell}
\]
and the linearization of the lifted Hamiltonian Reeb flow on $\C^{d}$ along
$\tgamma_\ell$ is the linear flow generated by $\tR_\nu^\ell$. Note that, by
replacing $\eta_\ell$ with $-\eta_\ell$ if necessary, we can and will assume that
$b^\ell > 0$ for all $\ell = 1, \ldots, m$. We can now use Lemma~\ref{lem:main} 
to assert that
\[
\mu_{CZ} (\gamma_\ell^N) = \mu_{CZ}(\tgamma_\ell^N)
\]
for all $\ell = 1, \ldots, m$ and all iterates $N\in\N$.

To compute $\mu_{CZ}(\tgamma_\ell^N)$ note first that, since $R_\nu$ is
assumed to generate a dense $1$-parameter subgroup of the torus $\T^{n+1}$,
we have that the closure of the $1$-parameter subgroup of $\T^d$ generated
by $\tR^\ell_\nu$ is a torus of dimension $n+1$. That immediately implies
that the $n+1$ real numbers
\[
\left\{b_1^\ell, \ldots, b_n^\ell, b^\ell \right\}
\]
are $\Q$-independent. We can then use Example~\ref{ex:CZ2} and the Direct
Sum property of the Conley-Zehnder index, to conclude that
\begin{align}
\mu_{CZ}(\tgamma_\ell^N) & = 
\sum_{i=1}^n \left(2 \left\lfloor N \frac{b_i^\ell}{b^\ell} \right\rfloor +1\right) +
2N \left( \sum_{j=1}^d (\teta_\ell)_j \right) \notag \\
& = 2 \left( \sum_{i=1}^n \left\lfloor N \frac{b_i^\ell}{b^\ell} \right\rfloor +
N \left( \sum_{j=1}^d (\teta_\ell)_j \right) \right) + n \notag \\
& = \text{even} + n \,.\notag
\end{align}
This implies that the contact homology degree is given by
\[
\deg (\tgamma_\ell^N) = \mu_{CZ}(\tgamma_\ell^N) + n - 2 =
\text{even} + n + n - 2 = \text{even}\,,
\]
which finishes the proof of Theorem~\ref{thm:main1}.

\section{Examples and proof of Theorem~\ref{thm:main2}} 
\label{s:examples}

\subsection{A particular family of good moment cones}

Let $\left\{e_1, e_2, e_3\right\}$ be the standard basis of $\R^3$. For each 
$k\in\N_0$ consider the cone $C(k)\subset\R^{3}$ with $4$ facets defined by the
following $4$ normals:
\begin{align}
\nu_1 & = e_1 + e_{3} = (1,0,1) \notag \\
\nu_2 & = - e_2 + e_{3} = (0, -1, 1) \notag \\
\nu_3 & = k e_2 + e_{3} = (0, k, 1) \notag \\
\nu_4 & = - e_1 + (2k-1) e_2 + e_{3} = (-1, 2k-1, 1) \notag
\end{align}
Each of these cones is good, hence defines a smooth, connected, closed toric 
contact $5$-manifold $(N_k, \xi_k)$. Because all the normals have last coordinate equal to one, 
Remark~\ref{rmk:c_1} implies that the first Chern class of all these contact manifolds is zero. 
Moreover, one can use Proposition~\ref{prop:pi1} to easily check that $N_k$ is simply connected for
all $k\in\N$. In fact, this family of good cones is $SL(3,\Z)$ equivalent to
the family of moment cones associated to the Sasaki-Einstein toric manifolds
$Y^{p,q}$, with $q=1$ and $p=k+1$, constructed by Gauntlett, Martelli, Sparks 
and Waldram in~\cite{GW1} (see also~\cite{MSY}). Hence we have that
\[
(N_k, \xi_k) \cong (S^2 \times S^3, \xi_k)\quad\text{with}\quad c_1 (\xi_k) = 0
\]
and, as hyperplane distributions, the $\xi_k$'s are all homotopic to each other.

When $k=0$ there is a direct way of identifying the toric contact manifold
$(N_0, \xi_0)$. In fact, the cone $C(0)\subset\R^3$ is $SL(3,\Z)$ equivalent to the
cone $C'\subset\R^3$ defined by the following $4$ normals:
\begin{align}
\nu'_1 & = e_1 = (1,0,0) \notag \\
\nu'_2 & = - e_2 + e_{3} = (0, -1, 1) \notag \\
\nu'_3 & = e_{2} = (0, 1, 0) \notag \\
\nu'_4 & = - e_1 + e_{3} = (-1, 0, 1) \notag
\end{align}
One easily checks that $C'$ is the standard cone over the square 
$[0,1]\times [0,1] \subset\R^2$. Hence, $(N_0, \xi_0)$ can be described as the
Boothby-Wang manifold over $(S^2 \times S^2, \omega = \sigma \times \sigma)$, 
where $\sigma (S^2) = 2\pi$. This is also the unit cosphere bundle of $S^3$ and
its Calabi-Yau symplectic cone is known in the physics literature as the conifold.

\begin{remark}
Gauntlett-Martelli-Sparks-Waldram construct in~\cite{GW2} a family of higher 
dimensional generalizations of the manifolds $Y^{p,q}$. They do not describe 
their exact diffeomorphism type and they do not write down the associated moment 
cones. The latter are described in~\cite{Ab} and can be used to show that, contrary to
what happens in dimension five, different cones in this higher dimensional family give 
rise to non-diffeomorphic manifolds.
\end{remark}

\subsection{Contact homology computations}

We will now apply the algorithm of section~\ref{s:proof1} to this family of good moment 
cones: $C(k)\subset\R^{3}$, $k\in\N_0$. We will do it for two different types of Reeb
vector fields.

First, we consider the case when the Reeb vector field $R_\nu \in\Xx (S^2 \times S^3, \xi_k)$ 
is induced by a Lie algebra vector $\nu\in\ft^3 \cong \R^3$ of the form
\[
\nu = (a_1, a_2, a_3) \approx (0,0,1)\,,
\]
with the $a_i$'s $\Q$-independent.

\begin{remark} \label{rmk:k=0}
When $k>0$, these vectors satisfy the requirement of Proposition~\ref{prop:sasaki} 
because the vector $(0,0,1)$ can be written as a positive linear combination of the
normals to $C(k)$:
\[
\frac{1}{3k+2} (\nu_1 + (3k-1) \nu_2 + \nu_3 + \nu_4) = (0,0,1)\,.
\]
When $k=0$, the second coordinate of all the normals is either zero or negative and
so we must have $a_2 < 0$.
\end{remark}

Each cone $C(k)$ has four edges:
\begin{itemize}
\item[(1)] The edge $E_1$, with $\gamma_1$ the corresponding simple closed
$R_\nu$-orbit, is the intersection of the facets $F_1$ and $F_3$ with normals
\[
\nu_1 = (1,0,1) \quad\text{and}\quad \nu_3 = (0,k,1)\,.
\]
The vector $\eta_1\in\Z^3$ can be chosen to be
\[
\eta_1 = \nu_4 = (-1, 2k-1, 1)\,.
\]
In fact, $\{\nu_1,\nu_3, \eta_1=\nu_4\}$ is a $\Z$-basis of $\Z^3$ and
\[
R_\nu = b^1_1 \nu_1 + b^1_2 \nu_3 + b^1 \eta_1
\]
with
\begin{align}
b^1_1 & = (1-k) a_1 - a_2 + k a_3 \notag \\
b^1_2 & = (2k-1) a_1 + 2 a_2 - (2k-1) a_3 \notag \\
b^1 & = - k a_1 - a_2 + k a_3 \notag
\end{align}
\item[(2)] The edge $E_2$, with $\gamma_2$ the corresponding simple closed
$R_\nu$-orbit, is the intersection of the facets $F_1$ and $F_2$ with normals
\[
\nu_1 = (1,0,1) \quad\text{and}\quad \nu_2 = (0,-1,1)\,.
\]
The vector $\eta_2\in\Z^3$ can be chosen to be
\[
\eta_2 = 2 \nu_3 - \nu_4 = (1, 1, 1)\,.
\]
In fact, $\{\nu_1,\nu_2, \eta_1=2\nu_3 - \nu_4\}$ is a $\Z$-basis of $\Z^3$ and
\[
R_\nu = b^2_1 \nu_1 + b^2_2 \nu_2 + b^2 \eta_2
\]
with
\begin{align}
b^2_1 & = 2 a_1 - a_2 - a_3 \notag \\
b^2_2 & = - a_1 + a_3 \notag \\
b^2 & = - a_1 + a_2 + a_3 \notag
\end{align}
\item[(3)] The edge $E_3$, with $\gamma_3$ the corresponding simple closed
$R_\nu$-orbit, is the intersection of the facets $F_3$ and $F_4$ with normals
\[
\nu_3 = (0,k,1) \quad\text{and}\quad \nu_4 = (-1,2k-1,1)\,.
\]
The vector $\eta_3\in\Z^3$ can be chosen to be
\[
\eta_3 = \nu_1 = (1, 0, 1)\,.
\]
In fact, $\{\nu_3,\nu_4, \eta_3=\nu_1\}$ is a $\Z$-basis of $\Z^3$ and
\[
R_\nu = b^3_1 \nu_3 + b^3_2 \nu_4 + b^3 \eta_3
\]
with
\begin{align}
b^3_1 & = (2k-1) a_1 + 2 a_2 - (2k-1) a_3 \notag \\
b^3_2 & = -k a_1 - a_2 + k a_3 \notag \\
b^3 & = (1-k) a_1 - a_2 + k a_3 \notag
\end{align}
\item[(4)] The edge $E_4$, with $\gamma_4$ the corresponding simple closed
$R_\nu$-orbit, is the intersection of the facets $F_2$ and $F_4$ with normals
\[
\nu_2 = (0,-1,1) \quad\text{and}\quad \nu_4 = (-1,2k-1,1)\,.
\]
The vector $\eta_4\in\Z^3$ can be chosen to be
\[
\eta_4 = 2\nu_3 - \nu_1 = (-1, 2k, 1)\,.
\]
In fact, $\{\nu_2,\nu_4, \eta_4 = 2 \nu_3 - \nu_1 \}$ is a $\Z$-basis of $\Z^3$ and
\[
R_\nu = b^4_1 \nu_2 + b^4_2 \nu_4 + b^4 \eta_4
\]
with
\begin{align}
b^4_1 & = a_1 + a_3 \notag \\
b^4_2 & = -(2k+1) a_1 - a_2 - a_3 \notag \\
b^4 & =  2k a_1 + a_2 + a_3 \notag
\end{align}
\end{itemize}

We can now compute the Conley-Zehnder index of all closed $R_\nu$ orbits, which
coincides in the $n=2$ case with the contact homology degree:
\begin{align}
\mu_{CZ} (\gamma_1^N) & = 2 \left\lfloor\frac{N}{k}\right\rfloor + \sign (a_1) +
\begin{cases}
1 &\text{if $N\neq$ multiple of $k$;} \\
\sign (a_2)  &\text{if $N=$ multiple of $k$.}
\end{cases} 
\notag \\
\mu_{CZ} (\gamma_2^N) & = 2 N + \sign (a_1) - \sign (a_2) \notag \\
\mu_{CZ} (\gamma_3^N) & = 2 \left\lfloor\frac{N}{k}\right\rfloor - \sign (a_1) +
\begin{cases}
1 &\text{if $N\neq$ multiple of $k$;} \\
\sign ((2k-1)a_1 + a_2)  &\text{if $N=$ multiple of $k$.}
\end{cases} 
\notag \\
\mu_{CZ} (\gamma_4^N) & = 2 N - \sign ((2k-1)a_1 + a_2) - \sign (a_1) \notag
\end{align}

To determine the rank of the contact homology groups, we can assume for example
that $a_1,a_2 < 0$ and get 
\begin{equation*}
\begin{array}{|c|c|c|c|c|c|c|} \hline
  \deg & \quad   0 \quad  & \quad 2 \quad & \quad 4 \quad & \quad 6 \quad & \quad 8 \quad & 
  \cdots
\\ 
\hline
\gamma_1   &     k      &    k    &   k  &  k & k & \cdots
\\ 
\hline 
\gamma_2   &     --      &    1    &   1  &  1 & 1 & \cdots
\\ 
\hline
\gamma_3  &     --      &    k    &   k  &  k & k & \cdots
\\ 
\hline 
\gamma_4  &     --     &    --   &   1  &  1 & 1 & \cdots
\\ 
\hline 
\rank &   k   & 2k+1 & 2k+2 & 2k+2 & 2k+2 & \cdots
\\
\hline
\end{array}
\end{equation*}
Another possibility would be to assume that $a_1>0$, $a_2 < 0$ and $(2k-1)a_1+a_2 < 0$. 
We would then get
\begin{equation*}
\begin{array}{|c|c|c|c|c|c|c|} \hline
  \deg & \quad   0 \quad  & \quad 2 \quad & \quad 4 \quad & \quad 6 \quad & \quad 8 \quad & 
  \cdots
\\ 
\hline
\gamma_1   &    --      &    k    &   k  &  k & k & \cdots
\\ 
\hline 
\gamma_2   &     --      &   --    &   1  &  1 & 1 & \cdots
\\ 
\hline
\gamma_3  &     k      &    k    &   k  &  k & k & \cdots
\\ 
\hline 
\gamma_4  &    --    &    1   &   1  &  1 & 1 & \cdots
\\ 
\hline 
\rank &   k   & 2k+1 & 2k+2 & 2k+2 & 2k+2 & \cdots
\\
\hline
\end{array}
\end{equation*}

Following a suggestion of Viktor Ginzburg, we will now consider a second type of
Reeb vector fields, namely those that are arbitrarily close to one of the normals of 
the cone $C(k)\in\R^3$.

More precisely, consider
\[
R_\nu = \sum_{i=1}^4 \eps_i \nu_i = (a_1, a_2, a_3)\,,
\]
which means that
\[
a_1 = \eps_1 - \eps_4\,,\quad a_2 = -\eps_2 + k \eps_3 + (2k-1) \eps_4
\quad\text{and}\quad a_3 = \eps_1 + \eps_2 + \eps_3 + \eps_4\,, 
\quad \eps_i > 0\,,\ i=1,\ldots, 4\,.
\]
Using the already determined formulas for $R_\nu$, we have that
\begin{itemize}
\item[(1)] On the edge $E_1$, where $\{\nu_1,\nu_3, \eta_1=\nu_4\}$ is the relevant 
$\Z$-basis, we can write
\[
R_\nu = (\eps_1 + (k+1)\eps_2) \nu_1 + (-(2k+1)\eps_2 + \eps_3) \nu_3 + ((k+1)\eps_2 + \eps_4) \eta_1\,.
\]
\item[(2)] On the edge $E_2$, where $\{\nu_1,\nu_2, \eta_2=2\nu_3 - \nu_4\}$ is the relevant 
$\Z$-basis, we can write
\[
R_\nu = (\eps_1 - (k+1)\eps_3 - 2(k+1)\eps_4) \nu_1 + (\eps_2 + \eps_3 + 2\eps_4) \nu_2 + 
((k+1)\eps_3 + (2k+1)\eps_4) \eta_2\,.
\]
\item[(3)] On the edge $E_3$, where $\{\nu_3,\nu_4, \eta_3 = \nu_1\}$ is the relevant 
$\Z$-basis, we can write
\[
R_\nu = (- 2(k+1)\eps_2 + \eps_3) \nu_3 + ((k+1)\eps_2 + \eps_4) \nu_4 + 
(\eps_1 + (k+1)\eps_2) \eta_3\,.
\]
\item[(4)] On the edge $E_4$, where $\{\nu_2,\nu_4, \eta_4 = 2\nu_3 - \nu_1\}$ is the relevant 
$\Z$-basis, we can write
\[
R_\nu = (2\eps_1 + \eps_2 + \eps_3) \nu_2 + (-2(k+1)\eps_1 - (k+1)\eps_3 + \eps_4) \nu_4 + 
((2k+1)\eps_1 + (k+1)\eps_3) \eta_4\,.
\]
\end{itemize}
We can now make $R_\nu$ arbitrarily close to a normal $\nu_j$ by considering
the $\eps_i$'s, with $i\ne j$, to be arbitrarily small positive numbers and $\eps_j \approx 1$.

Let us start with the case
\[
R_\nu \approx \nu_1\,.
\]
\begin{itemize}
\item[(1)] On the edge $E_1$ we have that
\[
R_\nu \approx \nu_1 + \eps \nu_3 + \eps \eta_1\,,
\]
with $\eps>0$ an arbitrarily small number. This implies that
\[
\mu_{CZ} (\gamma_1^N)  \approx \frac{2N}{\eps}
\]
can be made arbitrarily large for any $N\in\N$ and so $\gamma_1^N$ gives no contribution to
contact homology up to an arbitrarily large degree.
\item[(2)] The same happens for $\gamma_2^N$ since on the edge $E_2$ we have that
\[
R_\nu \approx \nu_1 + \eps \nu_2 + \eps \eta_2\,.
\]
\item[(3)] On the edge $E_3$ we have that
\[
R_\nu \approx \eps\nu_3 + \eps \nu_4 + \eta_3\,,
\]
with $\eps>0$ arbitrarily small. This implies that
\[
\mu_{CZ} (\gamma_3^N)  = 2N \quad\text{for}\quad N \approx 1, \ldots, \frac{1}{\eps}\,,
\]
and so $\gamma_3^N$ gives a rank $1$ contribution to contact homology in all positive
even degrees up to the arbitrarily large $1/\eps$.
\item[(4)] On the edge $E_4$ we have that
\[
R_\nu \approx (2+\eps)\nu_2 + (-2(k+1) + \eps) \nu_4 + (2k+1) \eta_4\,,
\]
with $\eps$ arbitrarily small. This implies a particularly interesting
behaviour for the Conley-Zehnder index of $\gamma_4^N$. In fact, when 
$m(2k+1)-2k \leq N \leq m(2k+1)$ for some $m\in\N$ and up to an arbitrarily large $N\in\N$,
we have that
\[
\mu_{CZ} (\gamma_4^N)  =
\begin{cases}
2m-2 & \text{if $m(2k+1) - 2k \leq N \leq m(2k+1) -k-1$;} \\
2m  & \text{if $m(2k+1) - k \leq N \leq m(2k+1) -1$;} \\
2m+2 & \text{if $N = m(2k+1)$.} 
\end{cases} 
\]
For $N\geq k+1$ this can also be written as
\[
\mu_{CZ} (\gamma_4^N)  =
\begin{cases}
2m  & \text{if $m(2k+1) - k \leq N \leq m(2k+1) -1$;} \\
2m+2 & \text{if $N = m(2k+1)$;} \\
2m & \text{if $m(2k+1) + 1 \leq N \leq m(2k+1) +k$;} 
\end{cases} 
\]
and we see that in this case the Conley-Zehnder index is not monotone with respect to $N$.
\end{itemize}
Hence, when $R_\nu \approx  \nu_1$, the rank of the contact homology groups is determined from
the following table:
\begin{equation*}
\begin{array}{|c|c|c|c|c|c|c|} \hline
  \deg & \quad   0 \quad  & \quad 2 \quad & \quad 4 \quad & \quad 6 \quad & \quad 8 \quad & 
  \cdots
\\ 
\hline
\gamma_1   &     --      &   --    &   --  &  -- & -- & \cdots
\\ 
\hline 
\gamma_2   &     --      &    --    &   --  &  -- & -- & \cdots
\\ 
\hline
\gamma_3  &     --      &    1    &   1  &  1 & 1 & \cdots
\\ 
\hline 
\gamma_4  &     k     &    2k   &   2k+1  &  2k+1 & 2k+1 & \cdots
\\ 
\hline 
\rank &   k   & 2k+1 & 2k+2 & 2k+2 & 2k+2 & \cdots
\\
\hline
\end{array}
\end{equation*}
In this case we have that all the interesting contact homology information is concentrated on just
one closed Reeb orbit (and its multiples): $\gamma_4$. 

When $R_\nu \approx \nu_4$ we obtain a similar picture, with all interesting contact homology information
concentrated on $\gamma_2$ and its multiples. In this case, and up to an arbitrarily large contact homology
degree, $\gamma_3^N$ and $\gamma_4^N$ contribute nothing, while $\gamma_1^N$ gives a rank $1$ 
contribution to degree $2N$.

When $R_\nu \approx \nu_2$ we have that $\gamma_2^N$ and $\gamma_4^N$ contribute nothing, while 
$\gamma_1^N$ and $\gamma_3^N$ contribute about half the rank of contact homology each. 
When $R_\nu \approx \nu_3$ we have that $\gamma_1^N$ and $\gamma_3^N$ contribute nothing, while 
$\gamma_2^N$ and $\gamma_4^N$ contribute about half the rank of contact homology each. 

In any case, and for any $k\in\N_0$, the final result is
\[
\rank HC_\ast (S^2  \times S^3, \xi_k ; \Q) =
\begin{cases}
k & \text{if $\ast = 0$;} \\
2k+1  & \text{if $\ast = 2$;} \\
2k+2 & \text{if $\ast > 2$ and even;} \\
0 & \text{otherwise.}
\end{cases} 
\]

\end{document}